\documentclass[oneside,12pt]{article}

\usepackage{amsmath}
\usepackage{amssymb}
\usepackage{pxfonts}
\usepackage{graphicx}
\usepackage{eucal}
\usepackage{mathrsfs}
\usepackage{theorem}
\usepackage{pifont}
\usepackage{sectsty}
\usepackage{amscd}
\usepackage{color}
\usepackage{fancyhdr}
\usepackage{framed}
\usepackage[all]{xy}
\usepackage[backref=page]{hyperref}
\usepackage{parskip}

\definecolor{shadecolor}{rgb}{0.8,0.8,0.8}

\theoremheaderfont{\fontfamily{pzc}\bfseries\large}
\newtheorem{theorem}{Theorem}[section]

\newtheorem{lemma}[theorem]{Lemma}
\newtheorem{proposition}[theorem]{Proposition}

\newtheorem{corollary}[theorem]{Corollary}

{\theorembodyfont{\rmfamily}

\newcommand{\specexercise}[1]{}

\newenvironment{proof}{{\flushleft \emph{Proof}:}}{\hfill\ding{110}}

\newenvironment{remark}{{\flushleft \fontfamily{pzc}\bfseries\large Remark:}}{}

\newcommand{\g}{\mathfrak{g}}

\newcommand{\h}{\mathfrak{h}}
\newcommand{\euc}{\mathfrak{e}}
\newcommand{\Vol}{\text{Vol}}
\newcommand{\Volg}{\text{Vol}_\g}
\newcommand{\VolgS}{\text{Vol}_{\g|_\S}}

\newcommand{\II}{{\text{II}}}
\newcommand{\M}{{\mathcal{M}}}

\newcommand{\TS}{{T\mathcal{S}}}
\newcommand{\NS}{{\mathcal{NS}}}
\newcommand{\N}{\mathcal{N}}
\renewcommand{\S}{\mathcal{S}}
\newcommand{\R}{\mathbb{R}}

\newcommand{\dist}{\operatorname{dist}}

\newcommand{\SO}[1]{\text{SO}(#1)}

\newcommand{\id}{\operatorname{Id}}
\newcommand{\sym}{\operatorname{Sym}}
\renewcommand{\skew}{\operatorname{Skew}}

\newcommand{\e}{\varepsilon}

\newcommand{\inner}[1]{\left\langle#1\right\rangle}      % \inner{.} => <.>

\newcommand{\pl}{\partial}

\newcommand{\Ga}{\Gamma}
\newcommand{\sig}{\sigma}

\newcommand{\weakly}[1]{\stackrel{#1}{\rightharpoonup}}

\newcommand{\Riem}{\mathcal{R}}

\newcommand{\qperp}{q^\perp}
\newcommand{\Pperp}{P^\perp}
\newcommand{\Ppar}{P^\parallel}

\def\Xint#1{\mathchoice
{\XXint\displaystyle\textstyle{#1}}%
{\XXint\textstyle\scriptstyle{#1}}%
{\XXint\scriptstyle\scriptscriptstyle{#1}}%
{\XXint\scriptscriptstyle\scriptscriptstyle{#1}}%
\!\int}
\def\XXint#1#2#3{{\setbox0=\hbox{$#1{#2#3}{\int}$ }
\vcenter{\hbox{$#2#3$ }}\kern-.6\wd0}}
\newcommand{\dashint}{\Xint-}

\newcommand{\teq}[1]{\stackrel{\mathrm{#1}}{=}}

\newcommand{\brk}[1]{\left(#1\right)}          % \brk{.}     => (.)
\newcommand{\Brk}[1]{\left[#1\right]}          % \Brk{.}     => [.]
\newcommand{\BRK}[1]{\left\{#1\right\}}        % \BRK{.}     => {.}
\newcommand{\Abs}[1]{\left| #1 \right|}        % \Abs{.}     => |.|

\newcommand{\beq}{\begin{equation}}
\newcommand{\eeq}{\end{equation}}

\newcommand{\propref}[1]{Proposition~\ref{#1}}
\newcommand{\lemref}[1]{Lemma~\ref{#1}}

\numberwithin{equation}{section}

\begin{document}

\title{On the role of curvature in the elastic energy of non-Euclidean thin bodies}
\author{Cy Maor\footnote{Department of Mathematics, University of Toronto.} \,and Asaf Shachar\footnote{Institute of Mathematics, The Hebrew University.}}
\date{}
\maketitle

\abstract{
We prove a relation between the scaling $h^\beta$ of the elastic energies of shrinking non-Euclidean bodies $\S_h$ of thickness $h\to 0$, and the curvature along their mid-surface $\S$.
This extends and generalizes similar results for plates \cite{BLS15,LRR15} to any dimension and co-dimension.
In particular, it proves that the natural scaling for non-Euclidean rods with smooth metric is $h^4$, as claimed in \cite{AAESK12} using a formal asymptotic expansion.
The proof involves calculating the $\Gamma$-limit for the elastic energies of small balls $B_h(p)$, scaled by $h^4$, and showing that the limit infimum energy is given by a square of a norm of the curvature at a point $p$.
This $\Gamma$-limit proves asymptotics calculated in \cite{AKMMS16}.
}

\tableofcontents

\setlength{\parindent}{0cm}
\addtolength{\parskip}{4pt}
%%%%%%%%%%%%%%%%%%%%%%%%%%%%%%%%%%%%%%%%%%%%%
\section{Introduction and main results}

%%%%%%%%%%%%%%%%%%%%%%%%
\subsection{Non-Euclidean elasticity}
\emph{Non-Euclidean}, or \emph{incompatible} elasticity is an elastic theory for bodies that do not have a reference configuration, i.e.~a stress-free configuration (therefore they are commonly referred to as \emph{pre-stressed} bodies). 
This theory has numerous applications --
it originated in the 1950's in the context of crystalline defects (see e.g.~\cite{Kon55,BBS55,Bs56}), and in recent years it is motivated by studies of growing tissues, thermal expansion, and other mechanics involving differential expansion or shrinkage \cite{AESK11, AAESK12, OY09, KES07, GSD16, AKMMS16}.

Mathematically, a pre-stressed elastic body is modeled as an $n$-dimensional compact, oriented Riemannian manifold $(\M^n,\g)$. It is "incompatible" if $\g$ is not flat.
Given a configuration $u:\M\to \R^n$, the elastic energy density at a point $p\in \M$ measures the \emph{strain} -- the discrepancy between the intrinsic metric $\g$ and the actual metric $u^\star \euc$ induced by the configuration ($\euc$ being the Euclidean metric in $\R^n$).
A prototypical "Hookean" energy is 
\beq
\label{eq:def_elastic_energy}
E_\M : W^{1,2}(\M;\R^n) \to \R,\qquad E_\M[u]:= \dashint_\M \dist^2(du, \SO{\g,\euc})\, d\Volg,
\eeq
where $\SO{\g,\euc}_p$ is the set of orientation preserving isometries $T_p\M\to \R^n$, and the distance is measured with respect to the inner-product norm on $T_p^*\M\otimes \R^n$ induced by $\g_p$ and the Euclidean metric $\euc$.
Representing all of the above in a positive orthonormal basis at $T_p\M$, $\SO{\g,\euc}_p$ and $\dist$ reduces to $\SO{n}$ and the Frobenius distance.
The notation $\dashint_\M$ means the integral normalized by the volume, that is $\dashint_\M f\, d\Volg:= \frac{1}{\Volg(\M)}\int_\M f d\Volg$; this will be important as we consider the elastic energies of a family of shrinking manifolds.

The definition of $E_\M$ suggest a second notion of incompatibility -- $(\M,\g)$ is incompatible if $\inf E_\M>0$ even in the absence of boundary conditions. 
In \cite[Theorem~2.2]{LP10} it was shown that this is equivalent to the first (geometric) notion of incompatibility -- $\inf E_\M =0$ if and only if $\Riem\equiv 0$, where $\Riem$ is the Riemann curvature tensor of $\g$ (see also \cite{KMS17} for a more general result between arbitrary manifolds).

Intuitively, one expect that the "more curvature" a body has, the less it is compatible with $\R^n$, and therefore the energy $E_{\M}$ would be higher.
A natural question is therefore to make the previous result quantitative -- to find a lower bound on the energy in terms of the curvature.
This problem is highly non-trivial. 
First, it is a global problem as it involves the entire geometry of the manifold.
second, $E_{\M}$ does not depend explicitly on the curvature, as the integrand involves only the metric $\g$ and not its derivatives.
The only general result we are aware of is \cite{KS12}, which gives a lower bound in terms of the scalar curvature for positively curved manifolds (and in dimension 2 for general manifolds). 
However, this bound is not very explicit, and in particular it is quite difficult to obtain from it effective bounds for \emph{thin elastic bodies}, which are the main focus of this paper.
These are described in the next section.

%%%%%%%%%%%%%%%%%%%
\subsection{Thin elastic bodies}
Much of the research in non-Euclidean elasticity, both in the physics and mathematics literature, is concerned with thin elastic bodies, i.e.~bodies that have one or more slender dimensions.
These include plate/shell theory and rod theory, corresponding to one and two slender dimensions (out of $3$), respectively.
The goal of these theories is to obtain the asymptotic behavior of the thin body as the thickness tends to zero.

Mathematically, the problem can be formulated as follows:
Let $(\M^n,\g)$ be a Riemannian manifold. For simplicity, assume that $\g$ is smooth (though for the results in this paper $C^2$ would suffice). 
Let $\S^k\subset\M^n$ be a compact $k$-dimensional oriented submanifold with Lipschitz boundary. $\S$ is the \emph{mid-surface} of the thin elastic body.
The thin elastic body $\S_h$ is the $h$-tubular neighborhood of $\S$ in $\M$.
More precisely, let $T\M|_\S = T\S \oplus \NS$ be the natural orthogonal decomposition, $\NS$ being the normal bundle of $\S$, and define
\beq
\label{eq:def_S_h}
\S_h := \BRK{ \exp_p(v) : p\in \S, v\in \NS, |v|\le h }.
\eeq

Two main (and interconnected) problems in the study of such bodies are finding the natural scaling of $\inf E_{\S_h}$ as $h\to 0$ (typically $\inf E_{\S_h} \sim h^\beta$ for some $\beta\ge 0$); and finding the limit of $\h^{-\beta} E_{\S_h}$ as $h\to 0$, which gives an effective elastic energy model for the mid-surface.
In the mathematics community, the last question is typically treated in the framework of $\Gamma$-convergence (based on the seminal results in the Euclidean case \cite{FJM02b,FJM06}).
We summarize below some of the main results in dimension reduction of non-Euclidean bodies that are relevant to this work (this does not aim to be a complete bibliography of the subject).

\paragraph{General dimension and codimension}
In \cite{KS14} a general $\Gamma$-convergence result was proved for any dimension and co-dimension, for the scaling $\beta = 2$. 
A corollary of their result is that $\inf E_{\S_h} = O(h^2)$ if and only if there exists $F\in W^{2,2}(\S;\R^n)$ and $\qperp\in W^{1,2}(\S;\NS^*\otimes \R^n)$ such that $dF\oplus \qperp\in \SO{\g,\euc}$.

\paragraph{Plates/shells ($n=3$, $k=2$)}
The case of plates and shells was initially treated in \cite{LP10,BLS15}, for the scaling $\beta = 2$.
Their results show that $\inf E_{\S_h} = O(h^2)$ if and only if $\S$ can be $W^{2,2}$ isometrically immersed in $\R^3$ (this is a special case of the results of \cite{KS14} mentioned above, in which the existence of $\qperp$ follows from the existence of the isometric immersion $F$). 
In \cite{BLS15,LRR15} it was shown, under the assumption that the metric $\g$ does not change along the thin dimension,
that $\inf E_{\S_h} = o(h^2)$ if and only if 
\[
|\Riem_{1212}| = |\Riem_{1213}| = |\Riem_{1223}| = 0,
\] 
where $\Riem$ is the curvature tensor of $\M$ and the first two coordinates parametrize the mid surface.
Furthermore, they proved that in this case $\inf E_{\S_h} = O(h^4)$, and that if $\inf E_{\S_h} = o(h^4)$ then the whole curvature tensor $\Riem\equiv 0$ on $\S$.
The assumption that $\g$ does not change along the thin dimension then implies that $\Riem\equiv 0$ everywhere, hence $\inf E_{\S_h} = 0$ in this case.

We also note that in \cite{LRR15} a complete $\Gamma$-convergence result for $h^{-4} E_{\S_h}$ is proved.
See also \cite{ALL17} for other recent results in the $O(h^2)$ regime, as well as numerous results in the physics literature for this scaling, e.g.~\cite{SRS07,ESK08,ESK09,ESK11}.
Other scalings can be obtained due to external forces \cite{BK14}, or singular metrics \cite{Olb17,COT17}, but these are further away from the context of this paper.

\paragraph{Rods ($n=3$, $k=1$)}
For the rod case, it was shown in \cite[Section~8.2]{KS14} that $\inf E_{\S_h} = o(h^2)$.
It was later shown, by an uncontrolled formal expansion, that one expects $\inf E_{\S_h} = O(h^4)$ for a general non-Euclidean rod \cite{AAESK12}.

Some recent results on non-Euclidean rods include \cite{CRS17,KO18}; in both of them the setting is slightly different from ours, which results in a natural energy scaling of $h^2$ (rather than $h^4$). This is due to external forces in \cite{CRS17} or rougher metrics in \cite{KO18}.

\paragraph{Other limits}
In \cite{AKMMS16} the case of a body which is thin in all dimensions was considered, which corresponds to the case $k=0$, i.e.~$\S=\{p\}$ (in this paper's framework); in other words, to the "local" elastic energy around a point.
There they show, by an uncontrolled formal expansion, that $\inf E_{\S_h} \sim h^4$, unless the Riemannian curvature at $p$ is zero.

When there are external forces or boundary conditions that imply that $\inf E_{\S_h} \sim 1$, the dimensionally-reduced limit is called the membrane limit. 
In the context of incompatible elasticity, a $\Gamma$-convergence derivation of the membrane limit for every dimension and codimension was obtained in \cite{KM14} (following the Euclidean case \cite{LR95,LR96}); this is further away from the context of this paper because of the stretching boundary conditions. 

%%%%%%%%%%%%%%%%%%%
\subsection{Main results}
In this paper we generalize the relations between curvature and energy scaling of thin plates \cite{BLS15,LRR15}, to every dimension and co-dimension. Our results provide a unifying ground for most of the results mentioned above.%without the assumption that the metric is "constant" throughout the thin dimensions.

We start by proving a $\Gamma$-convergence result for the energies of shrinking balls around a point; we later "lift" this result to a general submanifold $\S$. Let $B_h(p)$ denote the ball of radius $h$ around a point $p\in \M$. We show the functionals $h^{-4} E_{B_h(p)}$ $\Gamma$-converge to the functional
\beq
\label{eq:limiting_functional}
	I_\Riem: W^{1,2}(B,\R^n)\to \R,\qquad 
	I_\Riem[f] = \dashint_B \Abs{\sym df - \frac{1}{6} \Riem_{kijl}x^k x^l }^2,
\eeq
where $\Riem_{kijl}$ are the components of the Riemann curvature tensor at $p$ for some choice of an orthonormal basis at $p$, $B$ is the unit ball in Euclidean space, and $\sym df$ is the symmetric gradient $(\sym df)_{ij} = \pl_i f^k\delta_{kj} + \pl_j f^k\delta_{ki}$.
Note that minimizing $I_\Riem$ is equivalent to a pure-traction linear elastic problem in the ball, with smooth body and traction forces (see \cite[Section~6.3]{Cia88}).
The exact formulation of the $\Gamma$-convergence result is given in Theorem~\ref{thm:Gamma_convergence}, after introducing some required notations.

Using this $\Gamma$-convergence result, we prove the following theorem:
\begin{theorem}
\label{thm:curvature_norm_asymptotic}
\beq
\label{eq:thm:curvature_norm_asymptotic}
\lim_{h\to 0}  \frac{1}{h^4} \inf E_{B_h(p)} = |\Riem_p|^2,
\eeq
where $| \cdot |$ is an inner-product induced norm on the subspace of $(T_p^*\M)^3 \otimes T_p\M$ containing the possible curvature tensors at $p$. This norm is defined, in normal coordinates centered at $p$, as $|\Riem| := \sqrt{\min I_{\Riem}}$, where $I_{\Riem}$ is defined in \eqref{eq:limiting_functional}.
\end{theorem}

\begin{remark}
Note that $|\cdot|$, being an inner-product induced norm on a finite dimensional space, is of the form
$|\Riem_p|^2 = a^{ijklabcd}\Riem_{ijkl}\Riem_{abcd}$, 
where $\Riem_{ijkl}$ are the components of $\Riem_p$ in some orthonormal basis in $T_p\M$.
\eqref{eq:limiting_functional} implies that the constants $a^{ijklabcd}$ do not depend on $p$, and in this sense the norm is "point-independent".
In particular, the map $p \mapsto |\Riem_p|$ is continuous.
\end{remark}

Theorem~\ref{thm:Gamma_convergence} and Theorem~\ref{thm:curvature_norm_asymptotic} provide a "local" estimate of the infimal elastic energy in terms of the curvature.
Moreover, they prove the correctness of the formal asymptotics derived in \cite{AKMMS16}.

We then proceed to prove our main theorem regarding thin manifolds, thus establishing the relations between curvature and energy scaling of thin bodies in general dimension and co-dimension:
\begin{theorem}
\label{thm:energy_scaling_general}
\begin{enumerate}
	\item \cite{KS14}: 
	There exists $F\in W^{2,2}(\S;\R^n)$ and $\qperp\in W^{1,2}(\S;\NS^*\otimes \R^n)$ such that $dF\oplus \qperp\in \SO{\g,\euc}$ a.e.~if and only if
		\beq
		\label{eq:energy_O_h_squared}
		\inf E_{\S_h} = O(h^2).
		\eeq
	\item \beq
		\label{eq:energy_o_h_squared}
		\inf E_{\S_h} = o(h^2)
		\eeq
		if and only if there exist smooth maps $F:\S\to \R^n$ and $q^\perp: \NS\to \R^n$ such that $dF \oplus \qperp \in \SO{\g,\euc}$ and $\nabla q^\perp = -dF\circ \II_{\S,\M}$, where $\II_{\S,\M}$ is the second fundamental form (the shape operator) of $\S$ in $\M$.
		In particular, using appropriate identifications (given by $F$ and $q^\perp$),  $\II_{F(\S),\R^n}$ coincides with $\II_{\S,\M}$ 
		-- in this sense, the first and second forms of $\S$ satisfy the Gauss-Codazzi-Ricci equations in $\R^n$. 
		Moreover, \eqref{eq:energy_o_h_squared} implies that
		\beq
		\label{eq:energy_O_h_fourth}
		\inf E_{\S_h} = O(h^4).
		\eeq

		\item
                \eqref{eq:energy_o_h_squared} further implies that $\Riem^\M(X,Y) = 0$ for every $X,Y\in T\S$.\footnote{Note that this does not imply that $\S$ is flat, which is $\Riem^\S \equiv 0$.}
                If $\S$ is simply connected, then the converse also holds.
		\item  
		\beq
		\label{eq:infimum_energy_bound}
		\inf E_{\S_h} \ge ch^4 \dashint_S |\Riem^\M|^2 \,d\VolgS + o(h^4)
		\eeq
		where $|\Riem^\M|$ is a norm on the curvature, defined below in Theorem~\ref{thm:curvature_norm_asymptotic}, and $c$ is a universal constant.
		In particular, if 
		\beq
		\label{eq:energy_o_h_fourth}
		\inf E_{\S_h} = o(h^4),
		\eeq
		then $\Riem^\M|_\S \equiv 0$, that is $\Riem^\M(X,Y) = 0$ for every $X,Y\in T\M|_\S$.
		Furthermore, if \eqref{eq:energy_o_h_fourth} holds, $\S$ is simply-connected and $\Riem^\M$ is parallel along a foliation of curves emanating from $\S$, we have that for small enough $h$, $\S_h$ can be isometrically immersed in $\R^n$, hence $\inf E_{\S_h}=0$. 
\end{enumerate}
\end{theorem}

We note that in the physically-interesting special case of rods ($k=1$), Theorem~\ref{thm:energy_scaling_general} takes a particularly simple form:
\begin{corollary}
\label{cor:energy_scaling_rods}
If $\dim \S=1$, then $\inf E_{\S_h} = O(h^4)$. 
If $\inf E_{\S_h} = o(h^4)$, then $\Riem^\M|_\S \equiv 0$.
\end{corollary}
This proves the correctness of the scaling that appeared in \cite{AAESK12}.

%%%%%%%%%
Part 1 of the Theorem~\ref{thm:energy_scaling_general} is merely a restatement of a corollary of the main result of \cite{KS14}, which we include for completeness.
Parts 2 and 3 generalize the conditions for a scaling of $o(h^2)$ in \cite{BLS15,LRR15}; 
they clarify the geometric implications of this scaling also in the plate case.
These are proved by carefully analyzing the limit functional obtained in \cite{KS14}.
We prove part 4 by using Theorem \ref{thm:curvature_norm_asymptotic}; more accurately, we need a slightly stronger version of it, Theorem \ref{thm:curvature_norm_asymptotic_refined}, which allows for perturbations of the centers of the balls.

We note that the choice of the energy \eqref{eq:def_elastic_energy} is for the sake of simplicity alone; all the results and proofs will hold (with some natural adjustments) for a more general energy density $W:T^*\M\otimes \R^n\to [0,\infty)$ as long as $W$ is $C^2$ near $\SO{\g,\euc}$ and
\[
W|_{\SO{\g,\euc}} = 0, \quad W(A)\ge c\dist^2(A,\SO{\g,\euc}), \quad W(RA) = W(A), 
%\quad W(A\circ \Pi) = W(A),
\]
for some $c>0$, and every $R\in \SO{n}$. 

%%%%%%%%%%%%%%%
\paragraph{Open questions}
We list below several questions that arise in the context of this work, which are however not in of the scope of this paper; they will be considered in future works. 
\begin{enumerate}
\item The asymptotic analysis in \cite{AKMMS16} suggests that if one replaces $\R^n$ with a general ambient manifold, $h^{-4}\inf E_{B_h(p)}$ converges to a norm of the difference between the curvature at $p$ and the curvature at a point in the ambient manifold.
It would be very interesting to generalize Theorem~\ref{thm:curvature_norm_asymptotic} to this case.
\item In the last part of Theorem~\ref{thm:energy_scaling_general} we proved that $\Riem^\M|_\S \equiv 0$ is a necessary condition for the scaling $\inf E_{\S_h} = o(h^4)$.
We suspect that for a sufficient condition, one might also require that $\nabla \Riem^\M(X,Y) \equiv 0$ for $X,Y\in T\S$.
Obtaining a sufficient condition would require other tools than the ones used in this paper.
\item In this paper we only calculate the $\Gamma$-limit of $h^{-4} E_{\S_h}$ for the case where $\S$ is a point; for plates, this was done in \cite{LRR15}. A natural question is to calculate this for any dimension and codimension, in the spirit of the limit of $h^{-2} E_{\S_h}$ done in \cite{KS14}.
This would also give the exact limit of $h^{-4} \inf E_{\S_h}$ rather than the non-optimal bound \eqref{eq:infimum_energy_bound}, and will also answer question 2 above.
This general question seems, however, a pretty convoluted problem (even more than \cite{KS14}); a more approachable yet interesting partial result would be to prove this $\Gamma$-limit for non-Euclidean rods.
\end{enumerate}

%%%%%%%%%%%%%%%
\paragraph{Structure of this paper}
The paper is organized as follows: 
in Section~\ref{sec:local} we consider the "local" problem of dimension reduction of small balls. We first state the $\Gamma$-convergence result (Theorem~\ref{thm:Gamma_convergence}), and show that the scaling of $h^4$ is indeed the natural one (Section~\ref{sec:exp_energy}).
We then prove Theorem~\ref{thm:Gamma_convergence} and Theorem~\ref{thm:curvature_norm_asymptotic}.
In Section~\ref{sec:energy_scaling_general} we prove Theorem~\ref{thm:energy_scaling_general} through a sequence of lemmas; those in Section~\ref{sec:proofs_h_2} are more geometric and deal with the parts involving the $o(h^2)$ scaling;  those in Section~\ref{sec:proofs_h_4} are more analytic and deal with the $O(h^4)$ scaling.

%%%%%%%%%%%%%%%
\paragraph{Acknoledgments}
We thank Robert Jerrard for some useful advice and suggestions during the preparation of this paper, and Raz Kupferman for his critical reading of the manuscript. The second author was partially funded by the Israel Science Foundation (Grant No. 661/13), and by a grant from the Ministry of Science, Technology and Space, Israel and the Russian Foundation for Basic Research, the Russian Federation.

%%%%%%%%%%%%%%%%%%%%%%%%%%%%%%%%%%%%%%%%%%%%%
\section{$\Gamma$-limit of the elastic energy of shrinking balls}
%\section{Local elasticity at a point}
\label{sec:local}

This section is concerned with the "local" problem -- the $\Gamma$-convergence of elastic energies of small balls around a point (Theorem~\ref{thm:Gamma_convergence}) and the limit of their infima (Theorem~\ref{thm:curvature_norm_asymptotic}). As mentioned in the introduction, we shall prove a slightly stronger version of Theorem~\ref{thm:curvature_norm_asymptotic} which allows for perturbations (Theorem~\ref{thm:curvature_norm_asymptotic_refined} below): Instead of considering the behavior of $E_{B_h(p)} $, we shall consider the behavior of $E_{B_h(p_h)} $ where $p_h$ is a sequence in $\M$ converging to $p$.
We begin by introducing some notations.
\begin{itemize}
\item Fix $h_0< \text{inj}(p)$, so that $\exp_p : B_{h_0}(0)\subset T_p\M \to B_{h_0}(p)$ is a diffeomorphism, where $B_{h_0}(0)$ is the ball of radius $h_0$ centered at the origin in $ T_p\M$, and $B_{h_0}(p)$ is the ball of radius $h_0$ around $p$ in $\M$.
	For some small enough neighborhood $U$ of $p$, $\exp_q: B_{h_0}(0)\subset T_q\M \to B_{h_0}(q)$ is also a diffeomorphism for every $q\in U$, and the map $(q,v)\mapsto \exp_q(v)$ is smooth. 
\item Fix a smooth orthonormal frame $F$ of $T\M|_U$. 
For every $q\in U$, we identify $T_q\M \sim \R^n$ using $F_q$; in particular, this identifies $B_{h}(0)\subset T_q\M$ with $B_h(0)\subset \R^n$.
Using this identification, $\exp_q$ defines normal coordinates on $B_{h_0}(q)$.
Note that the components $\Riem_{ijkl}(q)$ of the Riemann curvature tensor in this coordinate system (centered at $q$) are the components of the curvature tensor with respect to $F$ at $q$. In particular,
the map $q\mapsto \Riem_{ijkl}(q)$ is smooth. 
\item For $q\in U$, denote $\iota_q:= \exp_{q}^{-1} : B_{h_0}(q) \to T_{q}\M\sim \R^n$; this is the identity map in the above normal coordinates (centered at $q$). With a slight abuse of notation we will consider $\iota_q$ also with a restricted domain $B_h(q)$ for some $h<h_0$.
\item For a map $u:B_h(q)\to \R^n$, define the rescaled map $\tilde{u}:B\to \R^n$ by $\tilde{u}(x) := u\brk{hx}$, where $B:=B_1(0)\subset \R^n$, using normal coordinates. Note that we view $\tilde{u}$ as a map between Euclidean spaces.
\item Unless otherwise noted, all integral norms (e.g.~$L^2$, $W^{1,2}$) are normalized by the volume of the relevant domain.
\end{itemize}

%%%%%%%%%
\begin{theorem}
\label{thm:Gamma_convergence}
Let $p_h\in \M$ be a sequence converging to $p$. Then the following hold:
\begin{enumerate}
\item Compactness and lower semicontinuity:
	Assume that $u_h\in W^{1,2}(B_h(p_h);\R^n)$ satisfy $E_{B_h(p_h)}[u_h] = O(h^4)$. Then
	\begin{enumerate}
	\item Rigidity: There exists $Q_h\in \SO{n}$ and $c_h\in \R^n$ such that the maps $\bar{u}_h = Q_hu_h -c_h$ satisfy $\|\bar{u}_h - \iota_{p_h}\|_{W^{1,2}(B_h(p_h);\R^n)} = O(h^2)$.
	\item Compactness: The ``displacements" $v_h = \bar{u}_h - \iota_{p_h}$ converge (modulo a subsequence), after rescaling, to some $f\in W^{1,2}(B,\R^n)$, in the following sense:\footnote{Note that for different choices of $Q_h$ we can have that $v_h$ converge to different functions; however we can further require that $\int_{B} f = 0$, $\int_{B} \skew(df) = 0$. In this case there is no ambiguity.
	%Proof sketch: Assume that $Q_h$ is a sequence associated with $u_h$ resulting in a limit $f$. Let $A\in \so(d)$, and define $Q_h' := e^{h^2A}Q_h\in \SO{\Cy{n}}$. 
%	Then $|Q_hdu_h - Q_h'du_h| \le |\id-e^{h^2A}||du_h| \le (h^2|A| + O(h^4))|du_h|$, hence $\bar{u}_h' = Q_h'u_h +c_h$ satisfies the same bounds as $\bar{u}_h$.
%	Define $v_h' = \bar{u}_h' - \iota$, then $dv_h' = dv_h' + (\id - e^{h^2A})d\iota +  (\id - e^{h^2A})(Q_h du_h-d\iota)$, hence $h^{-3}d\tilde{v}_h' \weakly{} df - A$.\\
%	For the other direction, assume $Q_h'$ is such that $\|\bar{u}_h' - \iota\|_{W^{1,2}(B_h;\R^\Cy{n})} = O(h^2)$, resulting in a function $g$. 
%	Then $dv_h' - dv_h = (Q_h'Q_h^T - \id)Q_hdu_h$, hence $h^{-3}(Q_h'Q_h^T- \id)Q_hd\tilde{u}_h\weakly{} dg-df$. Since $Q_hd\tilde{u}_h = O_{L^2}(h)$, we get that $Q_h'Q_h^T- \id = O(h^2)$.
%	Using this bound, we have $(Q_h'Q_h^T - \id)Q_hdu_h =  (Q_h'Q_h^T - \id)d\iota +  (Q_h'Q_h^T - \id)O_{L^2} (h^2) =  (Q_h'Q_h^T - \id)d\iota +  O_{L^2} (h^4)$, hence $h^{-3}(Q_h'Q_h^T- \id)Q_hd\tilde{u}_h\weakly{} dg-df$ actually implies $h^{-2}(Q_h'Q_h^T- \id) \to dg-df$. 
%	But the limit of $h^{-2}(Q_h'Q_h^T- \id)$ is a matrix in $\so(d)$, hence $dg = df + A$ for some constant $A\in \so(d)$.
	}
		\beq
		\label{eq:rescaled_convergence}
		\frac{1}{h^3}d\tilde{v}_h \weakly{} df \quad \text{weakly in $L^2$}.
		\eeq
	\item Lower semicontinuity: if $v_h\to f$ in the above sense, then
		\[
		\liminf \frac{1}{h^4} E_{B_h(p_h)}[u_h] \ge I_\Riem (f).
		\]
	\end{enumerate}
\item Recovery sequence: for every $f\in W^{1,2}(B,\R^n)$, there exists a sequence $u_h\in W^{1,2}(B_h(p_h);\R^n)$ such that $v_h = u_h - \iota_{p_h}$ converges strongly to $f$ (in the sense of \eqref{eq:rescaled_convergence}), and 
\[
\frac{1}{h^4} E_{B_h(p_h)}[u_h] \to I_\Riem (f).
\]
\end{enumerate}
\end{theorem}
%%%%%%%%

%%%%%%%%%%%%%%%%%%%%%%%%%%%%%%%%%%%%%%%%%%%%%
\subsection{The energy scaling of the exponential map}
\label{sec:exp_energy}
In this section we prove an upper bound of $\inf E_{B_h(p_h)}$, by using the exponential map. This yield the optimal scaling with $h$, though not the optimal constant.

\begin{lemma}[The asymptotic distortion of the exponential map]
\label{lm:asymptotic_exp}
For every $q\in U$, the inverse exponential map $\iota_{q}$ satisfies $E_{B_h(q)}[\iota_{q}] < Ch^4$ for some $C>0$ independent of $q$.
In particular, for the sequence $p_h\to p$ in Theorem~\ref{thm:Gamma_convergence}, $\inf E_{B_h(p_h)} = O(h^4)$.
\end{lemma}

\begin{proof}
The energy density $\dist (du,\SO{\g,\euc})$ satisfies 
\beq
\label{eq:using_sqrt_g}
\dist(du,\SO{\g,\euc}) = \dist (du\circ A^{-1}, \SO{n})
\eeq
for every $A\in \SO{\g,\euc}$, where in the right-hand side, the distance is with respect to the Frobenius norm on $\R^n\otimes \R^n$.
In particular, for an orientation preserving map $u$ we have
\beq
\label{eq:using_sqrt_g2}
\dist(du,\SO{\g,\euc}) = \Abs{ \sqrt{(du A^{-1})^T du A^{-1}} - \id},
\eeq
where the transpose on the right-hand side is the "standard" (Euclidean) transpose (since $du\circ A^{-1}:\R^n \to \R^n$).
We denote by $\g_q(x)$ the matrix representation of the metric $\g$ at a point $x$ with respect to the normal coordinates centered at $q$, and denote by $\sqrt{\g_q}(x)$ the positive square root of this matrix.
It is well known that $\sqrt{\g_q}\in \SO{\g,\euc}$, where both sides are evaluated at $x$.
Applying \eqref{eq:using_sqrt_g2} with $A = \sqrt{\g_q}$ and $u=\iota_q$, and using the fact that $\iota_q$ is the identity map in normal coordinates, we have that
\[
\dist(d\iota_q,\SO{\g,\euc}) = \Abs{ \sqrt{\g_q}^{-1} - \id}.
\]

In normal coordinates, we further have
\beq
\label{eq:g_coor}
(\g_q)_{ij}(x) = \delta_{ij} + \frac{1}{3} \Riem_{kijl}(q)x^k x^l + O(|x|^3),
\eeq
and therefore
\beq
\label{eq:sqrt_g_inverse_coor}
(\sqrt{\g_q}^{-1})^{ij}(x) = \delta^{ij} - \frac{1}{6} \Riem_{kijl}(q)x^k x^l + O(|x|^3).
\eeq
where $\Riem_{kijl}(q)$ are the components of the Riemannian curvature tensor at $q$.
Note that our choice of coordinates implies that the remainders $O(|x|^3)$ (and similar remainders below) can bounded independently of $q\in U$, that is $O(|x|^3)< C|x|^3$ for some $C>0$ independent of $q$.
Therefore we obtain
\[
\dist^2 (d\iota_q,\SO{\g_q,\euc}) = \Abs{(\sqrt{\g_q}^{-1})^{ij} - \delta^{ij}}^2 = \frac{1}{36}\delta^{ia}\delta^{jb}\Riem_{kijl}(q)\Riem_{cabd}(q)x^k x^lx^c x^d + O(|x|^5).
\]

The volume form in coordinates reads
\beq
\label{eq:dVol_coordinates}
d\Volg = \sqrt{\det(\g_q)}\,dx = (1+ O(|x|^{2}))dx.
\eeq

Plugging those expressions into the functional, and noting that the domain $B_h(q)$ is in normal coordinates the Euclidean ball $B_h(0)$, we obtain that
\beq
\label{eq:elastic_energy_exp}
\begin{split}
E_{B_h(q)}[\iota_q] &= \dashint_{B_h(0)} \brk{\frac{1}{36}\delta^{ia}\delta^{jb}\Riem_{kijl}(q)\Riem_{cabd}(q) x^k x^lx^c x^d + O(|x|^5)}\,dx \\
	&=  \delta^{ia}\delta^{jb}\kappa^{klcd} \Riem_{kijl}(q)\Riem_{cabd}(q)\, h^4 + O(h^5),
\end{split}
\eeq
where
\[
\kappa^{klcd} := \frac{1}{36}\,\dashint_{B_1(0)} x^k x^lx^c x^d\,dx.
\]
This estimate completes the proof, since $\Riem_{kijl}(q)$ can be bounded uniformly in $q$.
\end{proof}

\begin{remark}
The map $\iota_{q}$ is not optimal --- a direct calculation shows that by perturbing it one can get a lower $h^4$-coefficient than in \eqref{eq:elastic_energy_exp}.
Specifically, this can be done using $u_h(x) = x + P(x)$, where $P$ is a vector of homogeneous polynomials of degree $3$.
\end{remark}

%%%%%%%%%%%%%%%%%%%%%%%%%%%%%%%%%%%%%%%%%%%%%%%%%%
\subsection{Proof of Theorem~\ref{thm:Gamma_convergence} ($\Gamma$-convergence)}

In this section we prove Theorem~\ref{thm:Gamma_convergence}.
Throughout the proof, we will consider maps $A\in T_q^*\M\otimes \R^n$ for some $q\in B_h(p_h)$ (for example, $du_h(q)$ for $u_h\in W^{1,2}(B_h(p_h);\R^n)$).
As discussed before, $T_q^*\M\otimes \R^n$ has a natural inner-product induced by the metrics $\g$ and $\euc$, with respect to we can consider $|A|$, $\dist(A,\SO{\g,\euc})$, etc.

However, using the normal coordinates considered before, it would be useful to view $A$ also as a map $\R^n\to \R^n$, where both the domain and target are endowed with the Euclidean metric $\euc$.
Henceforth, whenever we say that we consider $A$ as a map $\R^n\to \R^n$, the norm we take is the Euclidean norm, and similarly we consider its distance (in $\R^n\otimes \R^n$) from $\SO{n}$.

By \eqref{eq:g_coor}, it follows that for every $A\in T_q^*\M\otimes \R^n$ and $q\in B_h(p_h)$ the metrics are equivalent with a uniform constant, that is
\beq
\label{eq:equivalence_of_norms}
\frac{|A|_{T_q^*\M\otimes \R^n}}{|A|_{\R^n\otimes \R^n}} = 1+ O(h^2).
\eeq
Therefore, in most cases it would not matter if we use $|A|_{T_q^*\M\otimes \R^n}$ or $|A|_{\R^n\otimes \R^n}$. In these cases, we simply write $|A|$. 
To simplify notation, we will also write $E_h$ instead of $E_{B_h(p_h)}$.

%%%%%%%%%%%%%%%%%%%%%%%%%%%%%%%%%%%%%%%%%%%%%%%%%%
\subsubsection{Rigidity (part 1a)}

The proof of this part is a direct application of the Friesecke-James-M\"uller rigidity theorem \cite[Theorem~3.1]{FJM02b}, taking into account that our metric is not Euclidean, but not far from it on small balls.

Let $u_h\in W^{1,2}(B_h(p_h);\R^n)$.
In normal coordinates centered at $p_h$, we can consider $u_h$ as a map $B_h(0)\to \R^n$ between Euclidean spaces.
By the Friesecke-James-M\"uller rigidity theorem \cite[Theorem~3.1]{FJM02b}, there exist a constant $C>0$ (independent of $u_h$ and $h$), and matrices $Q_h\in \SO{n}$ such that 
\[
\dashint_{B_h(0)} |Q_h du_h - \id |^2 \,dx \le C\dashint_{B_h(0)} \dist^2(du_h,\SO{n}) \,dx.
\]
Where distances and volume form are with respect to the Euclidean metric (not with respect to $\g$), as discussed above.
By \eqref{eq:dVol_coordinates}, we have that integrating with respect to $dx$ or $d\Volg$ is the same up to a multiplicative constant independent of $h$.
By \eqref{eq:equivalence_of_norms}, the $T_q^*\M\otimes \R^n$ and $\R^n\otimes \R^n$ norms on  $Q_h du_h - \id$ are equivalent, with a constant independent of $h$.
Using these, and the fact that $\iota_{p_h}$ is the identity map in coordinates, we can write the above inequality as
\beq
\label{eq:rigidity_aux_1}
\dashint_{B_h(p_h)} |d(Q_h u_h) - d\iota_{p_h} |^2 \,d\Volg \le C\dashint_{B_h(0)} \dist^2(du_h(x),\SO{n}) \,d\Volg(x).
\eeq
Note that the right-hand side is similar to $E_h[u_h]$, but not the same -- $\dist^2(du_h(x),\SO{n})$ is the distance squared of the coordinate representation of $du_h$ to $\SO{n}$ (in $\R^n\times \R^n$), while the integrand of $E_h[u_h]$ is the distance of $du_h$ to $\SO{\g,\euc}$ in $T^*\M\otimes \R^n$.
In order to complete the proof, we need to show the right-hand side is bounded by $C(E_h[u_h] + h^4)$, where $C>0$ is independent of $h$, $u_h$.

This follows from the following pointwise calculation. Let $q\in B_h(p_h)$ and let $T\in T_q^*\M\otimes \R^n$. Let $\hat{T}\in \R^n\otimes \R^n$ be the matrix representation of $A$ in normal coordinates.
We claim
\[
\Abs{ \dist(T,\SO{\g,\euc}) - \dist(\hat{T},\SO{n}) } \le C|T| h^2
\]
where each distance is considered with respect to its natural inner-product. The constant $C>0$ is independent of $q$ and $h$.
Indeed, using \eqref{eq:using_sqrt_g} and the fact that $S \to \dist(S,\SO{n})$ is $1$-Lipschitz (for maps $\R^n\to \R^n$), we have 
\[
\begin{split}
\Abs{ \dist(T, \SO{\g,\euc}) - \dist(\hat{T},\SO{n}) } 
	&= \Abs{ \dist(T\circ \sqrt{\g}^{-1},\SO{n}) - \dist(\hat{T},\SO{n})} \\
	&\le \Abs{ \hat{T}}\Abs{ \sqrt{g}^{-1} - \id } \\
	&\le C\Abs{ T}h^2,
\end{split}
\]
where in the last line we used \eqref{eq:sqrt_g_inverse_coor} and \eqref{eq:equivalence_of_norms}, centered at the point $p_h$.
We therefore have
\[
\begin{split}
\dashint_{B_h(p_h)} \dist^2(du_h,\SO{n}) \,d\Volg 
	&\le \dashint_{B_h(p_h)} (\dist(du_h,\SO{\g,\euc})+C|du_h|h^2)^2 \,d\Volg \\
	&\le \dashint_{B_h(p_h)} C'(\dist(du_h,\SO{\g,\euc})+ h^2)^2 \,d\Volg \\
	&\le 2C'\brk{E_h[u_h] + h^4}.
\end{split}
\]
Together with \eqref{eq:rigidity_aux_1}, this shows that
\[
\dashint_{B_h(p_h)} |d(Q_h u_h) - d\iota |^2 \le C(E_h[u_h] + h^4),
\]
for some constant $C>0$. Part 1a of Theorem~\ref{thm:Gamma_convergence} now follows by Poincar\'e inequality.
 
%%%%%%%%%%%%%%%%%%%%%%%%%%%%%%%%%%%%%%%%%%%%%%%%%%%%
\subsubsection{Compactness and lower bound (parts 1b and 1c)}
  
Suppose $E_h[u_h] = O(h^4)$ and let $\bar{u}_h = Q_hu_h - c_h$ as in part 1a of Thereom~\ref{thm:Gamma_convergence}, such that $\|v_h\|_{W^{1,2}(B_h(p_h);\R^n)} = O(h^2)$, where $v_h = \bar{u}_h - \iota$.
Let $\tilde{v}_h\in W^{1,2}(B; \R^n)$ be the rescaling of $v_h$, that is $\tilde{v}_h(x) := v_h\brk{hx}$.
Note that $\|d\tilde{v}_h\|_{L^2} = O(h^3)$ (recall that the norms are normalized by the volume of the domain, and that the Euclidean and Riemannian norms are uniformly equivalent by \eqref{eq:equivalence_of_norms} ).
Therefore we have that
\beq
\label{eq:eq_conv}
\frac{1}{h^3} d\tilde{v}_h \weakly{} V \quad \text{ in } L^2(B;T^*B\otimes \R^n).
\eeq
Note that $d(d\tilde{v}_h) = 0$ in $W^{-1,2}(B;\Lambda^2T^*B\otimes \R^n)$, by the weak Poincar\'e Lemma \cite[Theorem 6.17-4]{Cia13}, hence also $d(h^{-3} d\tilde{v}_h)=0$.
Since the weak convergence in $L^2(B;T^*B\otimes \R^n)$ respects the weak formulation of the $d$ operator, we obtain that $dV = 0$ in $W^{-1,2}$.
Invoking the Poincar\'e Lemma again, we obtain that $V=df$ for some $f\in W^{1,2}(B;\R^n)$.
This completes the proof of part 1b (compactness).

We now prove part 1c, the lower bound for the energy.
First, we write the energy density as
\beq
\label{eq:dist_expressed_with_G_h}
\dist(du_h, \SO{\g,\euc}) = \dist(du_h\sqrt{\g_{p_h}}^{-1}, \SO{n}) = \dist(\id + h^2 G_h,\SO{n}),
\eeq
where
\[
G_h\in L^2(B_h(p_h);\R^n\otimes \R^n), \qquad G_h := \frac{du_h\circ \sqrt{\g_{p_h}}^{-1} - \id}{h^2}.
\] 
Now, in coordinates we have (using $du_h = \id + dv_h$)
\beq
\label{eq:G_h_aux_1}
\frac{du_h\circ \sqrt{\g_{p_h}}^{-1} - \id}{h^2}
= \frac{( \sqrt{\g_{p_h}}^{-1} - \id)}{h^2}+\frac{dv_h}{h^2} + \frac{dv_h}{h^2}(\sqrt{\g_{p_h}}^{-1} - \id) 
\eeq
Since $\| \sqrt{\g_{p_h}}^{-1} - \id\|_\infty = O(h^2)$ and $\|v_h\|_{W^{1,2}(B_h(p_h);\R^n)} = O(h^2)$, we have $\|G_h\|_2 = O(1)$.
Let $\tilde{G}_h\in L^2(B;\R^n\otimes \R^n)$ be the rescaling of $G_h$, that is $\tilde{G}_h(x) = G_h(hx)$.
Since $\|G_h\|_2 = O(1)$ we also have $\|\tilde{G}_h\|_2 = O(1)$, hence $\tilde{G}_h$ weakly convergens in $L^2(B;\R^n\otimes\R^n)$ to some $G$.
From \eqref{eq:G_h_aux_1}, \eqref{eq:sqrt_g_inverse_coor} and \eqref{eq:rescaled_convergence}
a direct calculation shows that 
\beq
\label{eq:G_in_terms_of_f_and_R}
G(x) = df(x) - \frac{1}{6} \Riem_{kijl}(p)x^k x^l,
\eeq
using the continuity of $\Riem_{kijl}(p_h) \to \Riem_{kijl}(p)$.

Now, by Taylor expanding $\dist(\id+A,\SO{n})$, it follows from \eqref{eq:dist_expressed_with_G_h} that
\beq
\label{eq:Taylor_expend_G_h}
\Abs{\dist^2(du_h, \SO{\g,\euc}) -h^4\Abs{\frac{G_h+G_h^T}{2}}^2 } \le  \omega(h^2|G_h|).
\eeq
where $\omega(t)$ is a non-negative function satisfying $\lim_{t\to 0} \omega(t)/t^2 = 0$.
Therefore we have
\[
\begin{split}
\frac{1}{h^4}E_h(u_h) 
& \ge \dashint_{B_h(p_h)} \brk{\Abs{\frac{G_h+G_h^T}{2}}^2 - \frac{\omega(h^2|G_h|)}{h^4}}\,d\Volg \\
	&\ge \dashint_{B_h(p_h)} \chi_h \brk{\Abs{\frac{G_h+G_h^T}{2}}^2 - \frac{\omega(h^2|G_h|)}{h^4} } \,d\Volg\\ 
	&= \dashint_{B_h(p_h)} \brk{\chi_h\Abs{\frac{G_h+G_h^T}{2}}^2 - \chi_h |G_h|^2\frac{\omega(h^2|G_h|)}{h^4|G_h|^2}} \,d\Volg,
\end{split}
\]
where
\[
\chi_h (x) = 
\begin{cases}
1 & |G_h(x) |< h^{-1} \\
0 & |G_h(x) | \ge h^{-1}.
\end{cases}
\]
Now, on the support of $\chi_h$ we have $h^2|G_h|<h$, and therefore, since $\|G_h\|_2 = O(1)$, we have
\[
\dashint_{B_h(p_h)} \chi_h |G_h|^2\frac{\omega(h^2|G_h|)}{h^4|G_h|^2}\,d\Volg=\frac{1}{\Vol(B_h(p_h))}\int_{G_h <h^{-1}}  |G_h|^2\frac{\omega(h^2|G_h|)}{h^4|G_h|^2} \le \|G_h\|_2 \sup_{t \in (0,h)}\frac{\omega(t)}{t^2} \to 0.
\]
Therefore,
\beq
\label{eq:eq_tech9}
\begin{split}
\liminf \frac{1}{h^4}E_h(u_h)  
	&\,\,\ge \liminf \dashint_{B_h(p_h)} \chi_h\Abs{\frac{G_h+G_h^T}{2}}^2\, d\Volg \\
	&\,\, = \liminf \dashint_{B_h(p_h)} \Abs{\frac{\chi_h G_h+ \chi_h G_h^T}{2}}^2 \,d\Volg\\
	& \teq{\eqref{eq:dVol_coordinates}}  \liminf \dashint_{B_h(0)} \Abs{\frac{\chi_h G_h+ \chi_h G_h^T}{2}}^2 \,dx\\
	&\,\, = \liminf \dashint_{B} \Abs{\frac{\tilde{\chi}_h \tilde{G}_h+ \tilde{\chi}_h \tilde{G}_h^T}{2}}^2 \, dx.
\end{split}
\eeq
Since $\|G_h\|_2 = O(1)$, we have that $\tilde{\chi}_h \to 1$ in $L^2$ (and uniformly bounded), and therefore $\tilde{G}_h\weakly{} G$ implies that $\tilde{\chi}_h \tilde{G}_h\weakly{} G$.
 
By passing to subsequences, we can always assume that $\tilde{G}_h\weakly{} G$ for a subsequence that achieves $\liminf \frac{1}{h^4}E_h(u_h)$. Therefore, by the lower semicontinuity of the norm under weak convergence, \eqref{eq:eq_tech9} implies
\[
\liminf \frac{1}{h^4}E_h(u_h)  
	\ge \dashint_{B} \Abs{\frac{G+ G^T}{2}}^2\,dx \teq{\eqref{eq:G_in_terms_of_f_and_R}}  I_\Riem(f).
\]

%%%%%%%%%%%%%%%%%%%%%%%%%%%%%%%%%%%%%%%%%%%%%%%%%%%%
\subsubsection{Upper bound (part 2)}

We now prove part 2 of Theorem~\ref{thm:Gamma_convergence} -- for every $f\in W^{1,2}(B,\R^n)$, there exists a sequence $u_h\in W^{1,2}(B_h(p_h);\R^n)$ such that $v_h = u_h - \iota_{p_{h}}$ converges strongly to $f$ (in the sense of \eqref{eq:rescaled_convergence}), and $h^{-4} E_h[u_h] \to I_\Riem [f]$.

Indeed, fix $f\in W^{1,2}(B,\R^n)$, and choose $f_h\in W^{1,2}(B;\R^n)$ such that $f_h\to f$ and $\|df_h\|_\infty < h^{-1}$. 
Define, in coordinates centered at $p_h$, $u_h(x) = x + h^3 f_h(x/h)$.
Then obviously $v_h = u_h - \iota_{p_{h}} = h^3f_h(x/h)$ converges to $f$, and
\[
G_h(x) := \frac{du_h\circ \sqrt{\g_{p_h}}^{-1} - \id}{h^2} = \frac{1}{h^2}\brk{h^2df_h(x/h)-\frac{1}{6} \Riem_{kijl}(p_h)x^k x^l} + O(h).
\]
Therefore
\beq
\label{eq:G_h_recovery_seq}
\tilde{G}_h\to df(x)  -\frac{1}{6} \Riem_{kijl}(p)x^k x^l \quad \text{ strongly in $L^2$}.
\eeq
Now, since $\|G_h\|_\infty = O(h^{-1})$, we have from \eqref{eq:dist_expressed_with_G_h} and \eqref{eq:Taylor_expend_G_h} that
\[
\Abs{\dist^2(du_h, \SO{\g,\euc}) - h^4\Abs{\frac{G_h+G_h^T}{2}}^2} \le \omega(h^2|G_h|) = h^4|G_h|^2 \frac{\omega(h^2|G_h|) }{h^4|G_h|^2} \le |G_h|^2o(h^4),
\]
hence
\[
\Abs{\frac{1}{h^4}E_h(u_h) - \dashint_{B_h(p_h)} \Abs{\frac{G_h+G_h^T}{2}}^2} \le o(1)\dashint_{B_h(p_h)} |G_h|^2 = o(1),
\]
and by \eqref{eq:G_h_recovery_seq}
we obtain that $h^{-4}E_h[u_h]\to I_\Riem[f]$.

%%%%%%%%%%%%%%%%%%%%%%%%%%%%%%%%%%%%%%%%%%%%%%%%%%%%
\subsection{Proof of Theorem~\ref{thm:curvature_norm_asymptotic} (limit of infima)}

We shall now prove the slightly stronger version of Theorem~\ref{thm:curvature_norm_asymptotic}, namely:
 \begin{theorem}
\label{thm:curvature_norm_asymptotic_refined}
Let $p_h\in \M$ be a sequence converging to $p$. Then
\beq
\label{eq:thm:curvature_norm_asymptotic}
\lim_{h\to 0}  \frac{1}{h^4} \inf E_{B_h(p_h)} = |\Riem_p|^2,
\eeq
Where $|\Riem| := \sqrt{\min I_{\Riem}}$ is defined in normal coordinates centered at $p$.
\end{theorem}
So far we have shown that $h^{-4}E_h$ $\Gamma$-converges to $I_\Riem$, including a compactness argument.
In particular, a standard argument shows convergence of minimizers:
\begin{lemma}
\label{lm:convergence_of_min}
Let $u_h$ be a sequence of approximate minimizers of $\frac{1}{h^4} E_{B_h(p_h)}$, that is 
\[
 \frac{1}{h^4} E_{B_h(p_h)}[u_h] = \inf_{W^{1,2}(B_h(p_h);\R^n)} \frac{1}{h^4} E_{B_h(p_h)}+o(1).
\]
Then the associated displacements $v_h$ defined in Theorem~\ref{thm:Gamma_convergence} converge (modulo a subsequence) to a minimizer of $I_R$.
In particular,
\beq
\label{eq:conv_minimum}
\lim_{h\to 0} \inf_{W^{1,2}(B_h(p_h);\R^n)} \frac{1}{h^4} E_{B_h(p_h)} = \min_{W^{1,2}(B;\R^n)} I_\Riem.
\eeq
\end{lemma}

\begin{proof}
By Lemma~\ref{lm:asymptotic_exp}, $\inf \frac{1}{h^4} E_{B_h(p_h)} = O(1)$, hence $E_{B_h(p_h)}[u_h] = O(h^4)$.
Therefore, by Theorem~\ref{thm:Gamma_convergence}, parts 1(b) and 1(c), $v_h$ converges to $f\in W^{1,2}(B;\R^n)$.
Choose an arbitrary $f' \in W^{1,2}(B;\R^n)$, and let $u_h'$ be a recovery sequence for $f'$ according to part 2 of Theorem~\ref{thm:Gamma_convergence}.
We therefore have
\[
I_\Riem[f] \le \lim_{h\to 0} \frac{1}{h^4} E_{B_h(p_h)}[u_h] = \lim_{h\to 0} \inf_{W^{1,2}(B_h(p_h);\R^n)} \frac{1}{h^4} E_{B_h(p_h)} \le \lim_{h\to 0} \frac{1}{h^4} E_{B_h(p_h)}[u'_h] = I_\Riem[f'], 
\] 
hence $f$ is a minimizer. By choosing $f'=f$ in the above equation we obtain \eqref{eq:conv_minimum}.
\end{proof}

Therefore, in order to complete the proof of both Theorem~\ref{thm:curvature_norm_asymptotic} and Theorem~\ref{thm:curvature_norm_asymptotic_refined} we need to show that $N(\Riem) := \sqrt{\min I_{\Riem}}$ is a norm on $\Riem$.
Since $I_{\alpha\Riem}(\alpha f) = \alpha^2 I_{\Riem}(f)$ for every $\alpha\in \R$, and since we minimize over a vector space, we have
\[
N(\alpha\Riem) = |\alpha| N(\Riem).
\]
Note also that if $f_a$ is a minimizer of $I_{\Riem^a}$ for $a=1,2$, then
\[
\brk{\int_B \Abs{\sym (df_1 +df_2) - \frac{1}{6} (\Riem^1_{kijl}+ \Riem^2_{kijl})x^k x^l }^2}^{1/2} \le 
\sum_{a=1}^2 \brk{\int_B \Abs{\sym (df_a) - \frac{1}{6} \Riem^a_{kijl}x^k x^l }^2}^{1/2},
\]
hence
\[
N(\Riem^1 + \Riem^2) \le N(\Riem^1) + N(\Riem^2).
\]
Therefore $N$ is a semi-norm. 
A similar calculation shows that $2I_{\Riem^1}[f_1] + 2I_{\Riem^2}[f_2] = I_{\Riem^1+\Riem^2}[f_1+f_2] + I_{\Riem^1-\Riem^2}[f_1-f_2]$. This implies $f_1\pm f_2$ is a minimizer of $I_{\Riem^1\pm\Riem^2}$, so $N$ satisfies the parallelogram law.
Indeed, let $f_\pm$ be a minimizer of $I_{\Riem^1\pm \Riem^2}$, then
\[
\begin{split}
2I_{\Riem^1}[f_1] + 2I_{\Riem^2}[f_2] &= I_{\Riem^1+\Riem^2}[f_1+f_2] + I_{\Riem^1-\Riem^2}[f_1-f_2] \\
	&\ge I_{\Riem^1+\Riem^2}[f_+] + I_{\Riem^1-\Riem^2}[f_-] \\
	&= 2I_{\Riem^1}\Brk{\frac{f_++f_-}{2}} + 2I_{\Riem^2}\Brk{\frac{f_+-f_-}{2}} \\
	&\ge 2I_{\Riem^1}[f_1] + 2I_{\Riem^2}[f_2].
\end{split}
\]
Therefore, in order to complete the proof we need to show the positivity of $N$.

Denote $e_{ij} := \frac{1}{6} \Riem_{kijl}x^k x^l$. Since the minimizer of $I_\Riem$ exists, $N(\Riem) = 0$ if and only if there exists a function $f\in W^{1,2}(B;\R^n)$ such that $(\sym df)_{ij} = e_{ij}$.
The Saint-Venant lemma \cite[Section~6.18]{Cia13} implies that there exists such function if and only if
\[
\pl_{lj}e_{ik} + \pl_{ki}e_{jl} - \pl_{li}e_{jk} - \pl_{kj}e_{il} = 0.
\]
Note that
\[
\pl_{lj}e_{ik} = \frac{1}{6}\Riem_{aikb}\pl_{lj}(x^a x^b) = \frac{1}{6}\Riem_{aikb}(\delta_{al}\delta_{bj} + \delta_{aj}\delta_{bl}) = \frac{1}{6}(\Riem_{likj}+\Riem_{jikl}),
\]
hence (using the symmetries of the curvature tensor) we have
\[
\begin{split}
 6(\pl_{lj}e_{ik} + \pl_{ki}e_{jl} - \pl_{li}e_{jk} - \pl_{kj}e_{il}) 
	&= {\color{red}\Riem_{likj}}+{\color{blue}\Riem_{jikl}} + {\color{red}\Riem_{kjli}}+{\color{blue}\Riem_{ijlk}} - {\color{green}\Riem_{ljki}} -\Riem_{ijkl} - {\color{green}\Riem_{kilj}} -\Riem_{jilk} \\
	&=2({\color{red}\Riem_{likj}} - {\color{green}\Riem_{ljki}} + {\color{blue}\Riem_{jikl}} - \Riem_{ijkl}) \\
	&=2({\color{red}\Riem_{likj}} + {\color{green}\Riem_{ljik}} + 2{\color{blue}\Riem_{jikl}}) \\
	&=2(-\Riem_{lkji} + 2{\color{blue}\Riem_{jikl}}) \\
	&=6\Riem_{jikl}.
\end{split} 
\]
Therefore, the minimum energy is zero if and only if $\Riem=0$.
It follows that $N(\cdot)$ is a norm on the space of Riemannian curvatures at $p$.

%%%%%%%%%%%%%%%%%%%%%%%%%%%%%%%%%%%%%%%%%%%%%
\section{Energy scaling for general thin elastic bodies}
\label{sec:energy_scaling_general}

%%%%%%%
In this section we prove Theorem~\ref{thm:energy_scaling_general}. 
We begin by introducing some notations and by describing the main result of \cite{KS14}.
\begin{itemize}
\item Recall that  $T\M|_\S = T\S \oplus \NS$, and denote by $\Ppar_{\S}: T\M|_\S\to T\S$ and $\Pperp_{\S}: T\M|_\S \to \NS$ the orthogonal projections. The corresponding projections of other submanifolds are defined similarly.
\item We denote by $\pi_h:\S_h\to \S$, the natural projection $\pi_h(\exp_p(v)) := p$ (see \eqref{eq:def_S_h}).
\item We denote by $\nabla^\M$ the Levi-Civita connection on the tangent bundle of $\M$, and similarly for other manifolds.
	We denote by $\nabla^E$ the connection induced by the relevant Levi-Civita connection on a vector bundle $E$. For example, $\nabla^{\NS}$ is the connection of $\NS$ induced by $\nabla^\M$.
	We write $\nabla$ when the connection is clear from the context.
\item The second fundamental form (shape operator) of $\S$ in $\M$ is defined by
	\beq
	\label{eq:def_II}
	\II_{\S,\M}:\TS \times \NS \to \TS \, \, , \, \,   \II_{\S,\M} (v,\eta):=-\Ppar_{\S} (\nabla^{\M}_vN),
	\eeq
	where $N$ is a local extension of $\eta$ in the normal bundle $\NS$.
	The second fundamental form of other submanifolds is defined similarly.
\end{itemize}
The main result of \cite{KS14} is that the rescaled energies $h^{-2}E_{\S_h}$ $\Gamma$-converge (including a compactness statement), under an appropriate notion of $W^{1,2}$ convergence, to the limit energy
\[
E_\S : W^{2,2}(\S;\R^n) \times W^{1,2}(\S;\NS^*\otimes \R^n) \to [0,\infty]
\]
defined by
\[
E_\S(F,\qperp) := 
	\begin{cases}
	C\,\dashint_\S \brk{2|\Ppar_{\S}\circ q^{-1}  \circ \nabla q^\perp + \II_{\S,\M}|^2 + |\Pperp_{\S}\circ q^{-1} \circ \nabla q^\perp|^2 } \, d\VolgS & q\in \SO{\g,\euc} \text{  a.e.} \\
	\infty &\text{otherwise,}
	\end{cases}
\]
where $q := dF\oplus \qperp$, and $C$ is some constant depending on the codimension of $\S$ in $\M$. Note that \cite{KS14} and \eqref{eq:def_II} uses different sign conventions for $\II_{\S,\M}$, which results in a sign difference in the definition of $E_\S$.

\emph{Proof (of Theorem~\ref{thm:energy_scaling_general})}:
It follows immediately from the main result of \cite{KS14} described above that $\inf E_{\S_h} = O(h^2)$ if and only if $E_\S$ is not identically infinity, which implies that there exists $F\in W^{2,2}(\S;\R^n)$ and $\qperp\in W^{1,2}(\S;\NS^*\otimes \R^n)$ such that $dF\oplus \qperp\in \SO{\g,\euc}$ a.e.
This proves part 1 of Theorem~\ref{thm:energy_scaling_general}.

Furthermore, $\min E_\S = 0$ if and only if $\inf E_{\S_h} = o(h^2)$.
Note that the conditions for $E_\S(F,\qperp)$ to vanish, that is $-\Ppar_{\S}\circ q^{-1}  \circ \nabla q^\perp = \II_{\S,\M}$ and $\Pperp_{\S}\circ q^{-1} \circ \nabla \qperp = 0$, are equivalent to the condition $\nabla \qperp = -dF\circ \II_{\S,\M}$.

We split the analysis of the case $\inf E_{\S_h} = o(h^2)$, that is, of $\min E_\S = 0$, into several steps, details in lemmas bellows.
First, we prove in Lemma~\ref{lm:min_E_S_smooth} that if $\min E_\S = 0$, then the minimizer is smooth, which is used throughout the rest of the proof.
Next, in Lemma~\ref{lem:second_form} we show that the condition $-\Ppar_{\S}\circ q^{-1}  \circ \nabla q^\perp = \II_{\S,\M}$ implies that the second form $\II_{F(\S),\R^n}$ coincides with $\II_{\S,\M}$ under appropriate identifications that are detailed in the lemma. 

We then show, in Lemma~\ref{lem:normal_connection}, that the condition $\Pperp_{\S}\circ q^{-1} \circ \nabla \qperp = 0$ implies that the normal connection of $F(\S)$ in $\R^n$ coincides with that of $\S$ in $\M$ (again, under appropriate identifications).
Together with the identification of the second forms $\II_{F(\S),\R^n}$ and $\II_{\S,\M}$ (Lemma~\ref{lem:second_form}), this implies also that the covariant derivatives of the second fundamental forms coincide (Lemma~\ref{lem:everything_pullbacks}). 
Using this and the Gauss-Codazzi-Ricci equations, we conclude in Proposition~\ref{prop:curvature_condition_equivalence} that $\min E_\S = 0$ implies $\Riem^\M(X,Y) = 0$ for every $X,Y\in T\S$, and that for simply connected manifolds the converse also holds.
This completes the proof of part 3 of Theorem~\ref{thm:energy_scaling_general}.

The smoothness of the minimizer $(F,\qperp)$ of $E_\S$ in our case immediately shows that its recovery sequence $u_h\in W^{1,2}(\S_h;\R^n)$, as described in \cite[Section~6]{KS14}, satisfies $E_{\S_h}(u_h) < Ch^4$, which proves \eqref{eq:energy_O_h_fourth}. This is the content of Lemma~\ref{lm:o_h_2_O_h_4}.
This completes the proof of part 2 of Theorem~\ref{thm:energy_scaling_general}.

Finally, in Lemma~\ref{lm:o_h_4} we prove the bound \eqref{eq:infimum_energy_bound}.
The rest of part 4 of Theorem~\ref{thm:energy_scaling_general} immediately follows from that bound.
Indeed, assume that \eqref{eq:energy_o_h_fourth} holds, and $\Riem^\M$ is parallel along a foliation of curves emanating from $\S$.
Because of \eqref{eq:infimum_energy_bound}, assumption \eqref{eq:energy_o_h_fourth} implies that $\Riem^\M|_\S \equiv 0$ and the parallelism of $\Riem$ then implies that $\Riem^\M|_{\S_h} \equiv 0$. 
If $\S$ is simply-connected, then also $\S_h$, since they are homotopy equivalent for small enough $h$.  
A simply-connected $n$-dimensional manifold with zero curvature can be isometrically immersed in $\R^n$ \cite[Theorem~1.6-1]{Cia05}. Thus, $\min E_{\S_h} = 0$ (since we do not impose any boundary conditions or external forces).

%%%%%%%%%%%%%%%%%%%%%%%%%%%%%%%%%%%%%%%%%%%%%%%%%%%%%%
\subsection{Proofs regarding the scaling $\inf E_{\S_h} = o(h^2)$}

\label{sec:proofs_h_2}

In this section we prove our results concerning the scaling $\inf E_{\S_h} = o(h^2)$. 
These include most of part 2 and part 3 of Theorem~\ref{thm:energy_scaling_general}.
\eqref{eq:energy_O_h_fourth} in part 2, and part 4 of the theorem are proved in Section~\ref{sec:proofs_h_4}. 

\begin{lemma}
\label{lm:min_E_S_smooth}
If $\min E_\S = 0$, then the minimizer $(F,q^\perp)$ is smooth and unique up to rigid motions.  
\end{lemma}

\begin{proof}
We use the following notations: the indices $i,j,...$ are in the range $1..k$, the indices $a,b,...$ are in $k+1,...,n$, and the indices $I,J,...$ are in $1..n$.

Choose local coordinates $x^i$ on $\S$ and a frame $v_a$ for $\NS$. 
We extend the coordinate system to  a tubular neighborhood by choosing 
$x^a$, such that $\pl_a|_\S = v_a$.
Therefore, $\g_{ia}=0$ along $\S$. 
In these coordinates write $\qperp_a = \qperp(\pl_a)$.
Let $\Gamma_{IJ}^K$ be the Christoffel symbols of $(\M,\g)$ along $\S$. They are smooth functions of $x^i$.

Let $F\in W^{2,2}_{\text{iso}}(\S;\R^n)$ and $\qperp\in W^{1,2}(\S;N^*\S\otimes \R^n)$ satisfy $E_\S(F,\qperp) =0$. This implies the following
\begin{enumerate}
\item 
	\beq
	\label{eq:geometric_cond0}
	dF\oplus \qperp \in \SO{\g,\euc}\,\, \text{almost everywhere}.
	\eeq
\item For every $X\in T\S$ and $\eta\in \NS$,
		\beq
		\label{eq:geometric_cond1}
		-P_{\S}^{\parallel} \circ q^{-1} \big((\nabla_{X}^{\NS^* \otimes \R^n} q^{\perp})(\eta) \big) = \II_{\S,\M} (X,\eta).
		\eeq
	\item For every $X\in T\S$ and $\eta\in \NS$,
	\beq
	\label{eq:geometric_cond2}
	P^\perp_{\S} \circ q^{-1} \big((\nabla_{X}^{\NS^* \otimes \R^n} \qperp)(\eta)\big) = 0.
	\eeq
\end{enumerate}
Condition 1 implies that $\pl_iF\cdot \pl_j F = \g_{ij}$, $\pl_iF\cdot \qperp_a =0$ and $\qperp_a\cdot\qperp_b = \g_{ab}$, where $\cdot$ stands for the standard inner-product in $\R^n$.

Since $\{\pl_k F\}\cup\{\qperp_a\}$ is a basis to $\R^n$, we can write
\[
\pl_i\pl_j F = A_{ij}^l\pl_l F + A_{ij}^a \qperp_a
\]
for some functions $A_{ij}^l$, $A_{ij}^a$.

We now show that $A_{ij}^l = \Gamma_{ij}^l$ by repeating the calculation of the expression for the Christoffel symbols of the Levi-Civita connection on $\S$. Note that all the arguments below are valid in this Sobolev regularity, as they rely only on the validity of the product rule and on $\pl_i\pl_j=\pl_j\pl_i$, both of them hold in this regularity.
\[
\begin{split}
\pl_i\pl_j F \cdot \pl_l F &= \pl_i \g_{jl} - \pl_j F \cdot \pl_i\pl_l F \\
	&=\pl_i \g_{jl} - \pl_l\g_{ij} + \pl_l\pl_j F \cdot \pl_i F \\
	&=\pl_i \g_{jl} - \pl_l\g_{ij} + \pl_j\g_{li} - \pl_l F \cdot \pl_j\pl_i F,
\end{split}
\]
and therefore
\[
A_{ij}^m \g_{ml} = \pl_i\pl_j F \cdot \pl_l F = \frac{1}{2}\brk{\pl_i \g_{jl} - \pl_l\g_{ij} + \pl_j\g_{li}} = \Gamma_{ij}^m \g_{ml}.
\]
Up to now we have
\beq
\label{eq:F_ij_in_natural_frame_1}
\pl_i\pl_j F = \Gamma_{ij}^l\pl_l F + A_{ij}^a \qperp_a.
\eeq
Next, we consider conditions 2 and 3. By definition,
\beq
\label{eq:q_ai_in_natural_frame_1}
\begin{split}
(\nabla_{\pl_i}^{\NS^* \otimes \R^n} \qperp)(\pl_a) 
	&= \pl_i (\qperp (\pl_a)) - \qperp(\nabla_{\pl_i}^{\NS}\pl_a)\\
	&= \pl_i \qperp_a - \qperp(\Gamma_{ia}^b\pl_b)\\
	&= \pl_i \qperp_a - \Gamma_{ia}^b\qperp_b,
\end{split}
\eeq
where in the second line we used 
\[
\nabla_{\pl_i}^{\NS}\pl_a = P^\perp(\nabla_{\pl_i}^{\M}\pl_a) = P^\perp(\Gamma_{ia}^j\pl_j + \Gamma_{ia}^b\pl_b) = \Gamma_{ia}^b\pl_b.
\]
By  \eqref{eq:geometric_cond1} and \eqref{eq:q_ai_in_natural_frame_1} together with the identity
$
P_{\S}^{\parallel} \circ q^{-1}=q^{-1} \circ P_{F(\S)}^{\parallel},
$
we get
\beq
\label{eq:q_ai_in_natural_frame_2}
\begin{split}
P_{F(\S)}^\parallel (\pl_i \qperp_a) 
	&= -dF\circ \II_{\S,\M} (\pl_i,\pl_a) := dF \circ P_{\S}^{\parallel} (\nabla^{\M}_{\pl_i} \pl_a)\\
	&= dF\circ P_{\S}^{\parallel} (\Gamma_{ia}^j\pl_j + \Gamma_{ia}^b\pl_b)\\
	&= \Gamma_{ia}^j\pl_jF.\\
\end{split}
\eeq
Now, equations \eqref{eq:geometric_cond2}   and \eqref{eq:q_ai_in_natural_frame_1} yield
\beq
\label{eq:q_ai_in_natural_frame_3}
\begin{split}
0 &= P_{F(\S)}^\perp (\pl_i \qperp_a - \Gamma_{ia}^b\qperp_b) 
	= P_{F(\S)}^\perp (\pl_i \qperp_a) - \Gamma_{ia}^b\qperp_b.
\end{split}
\eeq
Combining \eqref{eq:q_ai_in_natural_frame_2} and \eqref{eq:q_ai_in_natural_frame_3} we obtain
\beq
\label{eq:q_ai_in_natural_frame_final}
\pl_i \qperp_a = P_{F(\S)}^\parallel (\pl_i \qperp_a) + P_{F(\S)}^\perp (\pl_i \qperp_a) = \Gamma_{ia}^j\pl_jF + \Gamma_{ia}^b\qperp_b.
\eeq
Using \eqref{eq:q_ai_in_natural_frame_final} we have
\[
\pl_i\pl_jF \cdot \qperp_a = \pl_i\g_{ja} - \pl_jF\cdot \pl_i\qperp_a = -\Gamma_{ia}^l\g_{lj},
\]
hence in \eqref{eq:F_ij_in_natural_frame_1} the coefficients $A_{ij}^a$ satisfy
\[
A_{ij}^a = -\Gamma_{ib}^l\g_{lj}\g^{ab}.
\]
Therefore the equation for $\pl_i\pl_j F$ is 
\beq
\label{eq:F_ij_in_natural_frame_final}
\pl_i\pl_j F = \Gamma_{ij}^l\pl_l F - \Gamma_{ib}^l\g_{lj}\g^{ab} \qperp_a.
\eeq
Since $\g_{ij}$, $\g^{ab}$ and $\Gamma_{IJ}^K$ are smooth functions, \eqref{eq:q_ai_in_natural_frame_final} and \eqref{eq:F_ij_in_natural_frame_final} show that $F$ and $\qperp$ are actually smooth, by a bootstrap argument.

Given the smoothness, the uniqueness follows from \cite[Section 3]{Ten71}, as explained in the proof of Lemma~\ref{prop:curvature_condition_equivalence} below. 
\end{proof}
%%%%%%%

%%%%%%%%%%%%%%%%%%%%%%%%%%%%%%%%%%

Next, we prove several lemmas leading to the proof of \propref{prop:curvature_condition_equivalence}. 
In the next two lemmas, we give a geometric interpretation of what does it mean for a pair $(F,\qperp)$ to satisfy $E_{\S}(F,\qperp)=0$.
Recall that $E_\S(F,\qperp)=0$ if and only if $q:=dF\oplus \qperp \in \SO{\g,\euc}$, and \eqref{eq:geometric_cond1} and \eqref{eq:geometric_cond2} hold, i.e.~$-\Ppar_{\S}\circ q^{-1}  \circ \nabla q^\perp = \II_{\S,\M}$ and $\Pperp_{\S}\circ q^{-1} \circ \nabla \qperp = 0$. %These are equivalent to the condition $\nabla \qperp = -dF\circ \II_{\S,\M}$.

In these lemmas we will repeatedly identify $\S$ and $F(\S)$, 
and therefore we can view $f:\S\to \R$ as a function on $F(\S)$. Under this identification $X(f) = dF(X)(f)$ for every $X\in T\S$, where in the right-hand side we consider $f$ as a function $F(\S)\to \R$.
This identification also extends to the trivial bundles $\S\times \R^n$ and $T\R^n|_{F(\S)} = F(\S) \times \R^n$.
Slightly abusing notation, we will denote the (trivial) connections on both bundles by $\nabla^{\R^n}$.
The identification $X(f) = dF(X)(f)$ extends (entry-wise) to $\nabla^{\R^n}$; 
namely, for $f:\S\to \R^n$ and $X\in T\S$, $\nabla_X^{\R^n} f = \nabla_{dF(X)}^{\R^n} f$, where in the right-hand side $f$ is considered as a section of $T\R^n|_{F(\S)}$.

\begin{lemma}[Equality of second fundamental forms]
\label{lem:second_form}
Assume $q \in \SO{\g,\euc}$.
$\II_{\S,\M}=-\Ppar_{\S} \circ q^{-1}\circ \nabla \qperp$ holds if and only if
\[
dF(\II_{\S,\M}(X,\eta))=\II_{F(\S),\R^n}(dF(X),q^{\perp}(\eta)) \, \, \text{for every } (X,\eta) \in \TS \times \NS.
\]
\end{lemma}
This lemma shows that $\II_{F(\S),\R^n}$ and $\II_{\S,\M}$ coincide, when we identify $\TS \cong dF(\TS),\NS \cong \N F(\S)$ using the maps $dF$ and $\qperp$, respectively.
Here $\N F(\S):=\brk{ dF\brk{\TS}}^{\perp}$ is the normal bundle to the image $F(S)$ in $\R^n$.

\begin{proof}
Let $(X,\eta) \in T_p\S \times \N_p\S$ and let $N$ be a local extension of $\eta$ normal to $\S$. Then, identifying the trivial bundle $\S\times \R^n$ with $T\R^n|_{F(\S)}$, and using  the identity
$P_{\S}^{\parallel} \circ q^{-1}=q^{-1} \circ P_{F(\S)}^{\parallel}$ (which holds since $q\in \SO{\g,\euc}$), we have
\[
\begin{split}
dF_p\circ P_{\S}^{\parallel} \circ q^{-1} \big((\nabla_{X}^{\NS^* \otimes \R^n} q^{\perp})(\eta) \big)
	&= P_{F(\S)}^{\parallel} \big((\nabla_{X}^{\NS^* \otimes \R^n} q^{\perp})(\eta) \big) \\
	&= P_{F(\S)}^{\parallel} \big(\nabla_{X}^{\R^n} (q^{\perp}(N)) -  \qperp( \nabla_X^{\NS} N)\big)\\
	&=P_{F(\S)}^{\parallel} \nabla_{X}^{\R^n} (q^{\perp}(N)),
\end{split}
\]
Hence $\II_{\S,\M}=-\Ppar_{\S} \circ q^{-1}\circ \nabla \qperp$ is equivalent to 
\[
dF_p(\II_{\S,\M}(X,\eta))= -P_{F(\S)}^{\parallel} \brk{\nabla^{\R^n}_{dF_p(X)} (q^{\perp}(N))}.
\]
On the other hand, the right-hand side of this equality is the definition of $\II_{F(\S),\R^n}(dF_p(X),q^{\perp}(\eta))$.
Therefore, we obtain,
\[
\II_{\S,\M}=-\Ppar_{\S} \circ q^{-1}\circ \nabla \qperp \iff dF(\II_{\S,\M}(X,\eta))=\II_{F(\S),\R^n}(dF_p(X),q^{\perp}(\eta)).
\]
\end{proof}

%%%%%%%%%%%%%%%%%%%%%%%%%%%%%%%%
\begin{lemma}[Equality of normal connections]
\label{lem:normal_connection}
Let $F,\qperp$ be smooth, and $q\in \SO{\g,\euc}$. Then
$P^{\perp}_\S \circ q^{-1} \circ \nabla q^{\perp}=0$ holds if and only if
\[
q^{\perp} (\nabla_X^{\NS} \sig)=\nabla^{NF(\S)}_{dF(X)}(q^{\perp}\sig)
\,\, \text{  for every  } \, X \in T\S \, \text{ and } \, \sig \in \Ga(\NS),
\]
where $\nabla^{NF(\S)}_{dF(X)}(q^{\perp}\sig) = P^{\perp}_{F(\S)} \bigg( \nabla_{X}^{\R^n}\big(q^{\perp}(\sig)\big)\bigg) $.
\end{lemma}
This lemma shows that the normal connections $\nabla^{\NS}$ and $\nabla^{NF(S)}$ coincide, under the identifications $\TS \cong dF(\TS),\NS \cong NF(S)$ induced by the maps $dF$ and $\qperp$, respectively.
\begin{proof}
Given $X \in \TS$ and  $ \sig \in \Ga(\NS)$ we have 
\[
\nabla_{X}^{\R^n}\big(q^{\perp}(\sig)\big)=\big(\nabla_{X}^{\NS \otimes \R^n}q^{\perp}\big)(\sig) + q^{\perp} \big( \nabla_{X}^{\NS} \sig \big),
\]
so
\[
P^{\perp}_\S \circ q^{-1} \bigg(\nabla_{X}^{\R^n}\big(q^{\perp}(\sig)\big)\bigg)=P^{\perp}_\S \circ q^{-1} \bigg(\big(\nabla_{X}^{\NS \otimes \R^n}q^{\perp}\big)(\sig)\bigg) +  \nabla_{X}^{\NS} \sig.
\]
Thus, $P^{\perp}_\S \circ q^{-1} \circ \nabla q^{\perp}=0$ holds if and only if
\[
P^{\perp}_\S \circ q^{-1} \bigg(\nabla_{X}^{\R^n}\big(q^{\perp}(\sig)\big)\bigg)= \nabla_{X}^{\NS}.
\]
Using $q^\perp \circ P_{\S}^{\perp} \circ q^{-1}= P_{F(\S)}^{\perp}$ (which holds since $q\in \SO{\g,\euc}$), we have that the above equation holds if and only if
\[
P^{\perp}_{F(\S)} \bigg( \nabla_{X}^{\R^n}\big(q^{\perp}(\sig)\big)\bigg)= q^{\perp}(\nabla_{X}^{\NS} \sig).
\]
\end{proof}

%%%%%%%%%%

%%%%%%%%%%
Next, we prove the final lemma required for establishing \propref{prop:curvature_condition_equivalence}. This lemma combines the previous two lemmas,  \ref{lem:second_form} and \ref{lem:normal_connection} and shows that the derivatives of the second fundamental forms coincide (again under the appropriate identifications).

In this lemma, we will use the following notation:
$B_{\S}:\TS \times \TS \to \NS$ is defined by $\inner{B_\S(X,Y),\eta} = \inner{\II_{\S,\M}(X,\eta),Y}$.
We also consider $B_\S$ as a map $\TS \times \TS \times \NS \to \R$, via $(X,Y,\eta) \mapsto \inner{B_\S(X,Y),\eta}$. 
Finally, we extend the covariant derivative to tensors of this type in the usual way, as follows:
\[
\nabla^\M_X B_\S(Y,Z,\eta)=X(B_{\S}(Y,Z,\eta))-B_{\S}(\nabla^{\S}_XY,Z,\eta)-B_{\S}(Y,\nabla^{\S}_XZ,\eta)-B_{\S}(Y,Z,\nabla^{\NS}_X\eta).
\]

%%%%%%%%%%
\begin{lemma}[Coincidence of the derivatives]
\label{lem:everything_pullbacks}
Let $(F,\qperp)$ satisfy \eqref{eq:geometric_cond0}-\eqref{eq:geometric_cond2}.
Then, for every $X,Y\in \Gamma(T\S)$ and $\eta\in \Gamma(\NS)$ the following hold:
\begin{enumerate}
\item
\beq
\label{eq:eq_tautology0}
q^{\perp} (B_\S(X,Y))=B_{F(\S)}(dF(X),dF(Y)).
\eeq
\item
\beq
\label{eq:eq_tautology-2}
B_\S(X,Y,\eta)=B_{F(\S)}(dF(X),dF(Y),\qperp(\eta)).
\eeq
\item
\beq
\label{eq:eq_tautology-1}
\nabla^\M_X B_\S(Y,Z,\eta) = \nabla^{\R^n}_{dF(X)} B_{F(\S)}(dF(Y),dF(Z),\qperp(\eta)).
\eeq
\end{enumerate}
\end{lemma}

\begin{proof}
\lemref{lem:second_form}, together with the fact that $F:\S\to F(\S)$ is an isometry, implies
\beq
\label{eq:eq_tautology1}
\begin{split}
\inner{B_\S(X,Y),\eta}&=\inner{\II_{\S,\M}(X,\eta),Y}=\inner{dF(\II_{\S,\M}(X,\eta)),dF(Y)}\\
&=\inner{\II_{F(\S),\R^n}\brk{dF(X),q^{\perp}(\eta)},dF(Y)}\\
&= \inner{B_{F(\S)}(dF(X),dF(Y)),q^{\perp}(\eta)}.
\end{split}
\eeq
Since $q^{\perp}:\NS\to \N F(\S)$ is an isometry,
\beq
\label{eq:eq_tautology2}
\inner{B_\S(X,Y),\eta}=\inner{q^{\perp}(B_\S(X,Y)),q^{\perp}(\eta)}.
\eeq
Combining \eqref{eq:eq_tautology1} and \eqref{eq:eq_tautology2} proves \eqref{eq:eq_tautology0} and \eqref{eq:eq_tautology-2}.

We now prove \eqref{eq:eq_tautology-1}. Using \eqref{eq:eq_tautology-2}, we get
\[
\begin{split}
\nabla^\M_X B_\S(Y,Z,\eta) &= X\bigg( B_{F(\S)}\big(dF(Y),dF(Z),\qperp(\eta)\big)\bigg)-B_{F(\S)}\big(dF(\nabla^{\S}_XY),dF(Z),q^{\perp}(\eta)\big)\\ 
&\quad - B_{F(\S)}\big(dF(Y),dF(\nabla^{\S}_XZ),q^{\perp}(\eta)\big)-B_{F(\S)}\big(dF(Y),dF(Z),q^{\perp}(\nabla^{\NS}_X\eta) \big).
\end{split}
\]
On the other hand,
\[
\begin{split}
\nabla^{\R^n}_{dF(X)} B_{F(\S)}(dF(Y),dF(Z),\qperp(\eta)) &= dF(X)\bigg( B_{F(\S)}\big(dF(Y),dF(Z),\qperp(\eta)\big)\bigg)-B_{F(\S)}\big(\nabla^{F(\S)}_{dF(X)}dF(Y),dF(Z),q^{\perp}(\eta)\big)\\ 
&\quad - B_{F(\S)}\big(dF(Y),\nabla^{F(\S)}_{dF(X)}dF(Z),q^{\perp}(\eta)\big)-B_{F(\S)}\big(dF(Y),dF(Z),(\nabla^{NF(\S)}_{dF(X)}(q^{\perp}\eta) \big).
\end{split}
\]
The first summand is the same by the identification of $\S$ and $F(\S)$ discussed before Lemma~\ref{lem:second_form}.
The second summand is the same since $dF(\nabla^{\S}_XY)=\nabla^{F(\S)}_{dF(X)}dF(Y)$ because $F:S \to F(S)$ is an isometry, hence preserves the connection.
The last two summands are the same by \lemref{lem:normal_connection}.
\end{proof}

%%%%%%%%%%%%%%%%%%%%%%%%%
Finally, we use the above to prove part 3 of Theorem~\ref{thm:energy_scaling_general}.
\begin{proposition}
\label{prop:curvature_condition_equivalence}
Let $\Riem^\M$ the Riemannian curvature tensor of $\M$. 
Assume there exists $F$ and $\qperp$ that satisfy equations \eqref{eq:geometric_cond0}-\eqref{eq:geometric_cond2}, then
\[
\Riem^\M(X,Y) = 0\,\qquad \forall X,Y\in T\S.
\]
If $\S$ is simply connected, then the converse holds.
Moreover, $F$ and $\qperp$ are unique up to a rigid motion.
\end{proposition}

%%%%%%%
\begin{proof}
In this proof, $X,Y,Z,T\in T\S$, and $\eta,\zeta\in N\S$. 
First assume the existence of such $F,\qperp$. 
The Gauss equation \cite[Chapter 6, Proposition 3.1]{doC92}, together with \eqref{eq:eq_tautology0} and the fact that $dF\oplus \qperp$ is an isometry, imply
\beq
\label{eq:Gauss}
\begin{split}
&\inner{\Riem^\M(X,Y)Z,T} = \inner{R^\S(X,Y)Z,T} - \inner{B_\S(Y,T),B_\S(X,Z)} + \inner{B_\S(X,T),B_\S(Y,Z)}\\
	&\quad= \inner{\Riem^{F(\S)}(dF(X),dF(Y))dF(Z),dF(T)} - \inner{B_{F(\S)}(dF(Y),dF(T)),B_{F^(\S)}(dF(X),dF(Z))}\\
	& \quad \quad+\inner{B_{F(\S)}(dF(X),dF(T)),B_{F(\S)}(dF(Y),dF(Z))}\\
	&\quad = \inner{\Riem^{\R^n}(dF(X),dF(Y))dF(Z),dF(T)} =0.
\end{split}
\eeq
Applying the Coddazi equation \cite[Chapter 6, Proposition 3.4]{doC92}, and using \eqref{eq:eq_tautology-1} we have that
\beq
\label{eq:Coddazi}
\begin{split}
\inner{\Riem^\M(X,Y)Z,\eta} &= \nabla^\M_Y B_\S(X,Z,\eta) - \nabla^\M_X B_\S(Y,Z,\eta)\\
	& = \nabla^{\R^n}_{dF(Y)} B_{F(\S)}(dF(X),dF(Z),\qperp(\eta)) - \nabla^{\R^n}_{dF(X)} B_{F(\S)}(dF(Y),dF(Z),\qperp(\eta)) \\
	& = \inner{\Riem^{\R^n}(dF(X),dF(Y))dF(Z),\qperp(\eta)} =0.
\end{split}
\eeq
\eqref{eq:Gauss} and \eqref{eq:Coddazi} together imply
\beq
\label{eq:R_XYZ_eq_0}
\Riem^\M(X,Y)Z = 0.
\eeq
Finally, the equality of the normal connections (\lemref{lem:normal_connection}) implies equality of the normal curvatures(i.e.~the curvature tensors associated with the normal connections)
\[
\inner{\Riem^\perp_{\S}(X,Y) \eta,\zeta} = \inner{\Riem^\perp_{F(\S)}(dF(X),dF(Y)) \qperp(\eta),\qperp(\zeta)},
\]
and therefore, using Ricci equation \cite[Chapter 6, Proposition 3.1]{doC92}, we have
\[
\begin{split}
\inner{\Riem^\M(X,Y)\eta,\zeta} 
	&= \inner{[\II_{\S,\M}(X,\eta),\II_{\S,\M}(X,\zeta)], Y} + \inner{\Riem^\perp_{\S}(X,Y) \eta,\zeta} \\
	&= \inner{[\II_{F(\S),\R^n}(dF(X),\qperp(\eta)),\II_{F(\S),\R^n}(dF(X),\qperp(\zeta))], dF(Y)} + \inner{\Riem^\perp_{F(\S)}(dF(X),dF(Y)) \qperp(\eta),\qperp(\zeta)} \\
	&=\inner{\Riem^{\R^n}(dF(X),dF(Y))\qperp{\eta},\qperp{\zeta}} = 0.
\end{split}
\]
The Codazzi equation \eqref{eq:Coddazi}, and the symmetries of $\R^\M$ also imply that
\[
\inner{\Riem^\M(X,Y)\eta,Z} = -\inner{\Riem^\M(X,Y)Z,\eta} =0,
\]
and therefore
\[
\Riem^\M(X,Y)\eta =0.
\]
Together with \eqref{eq:R_XYZ_eq_0}, this implies that
\[
\Riem^\M(X,Y) = 0.
\]
Now assume $\Riem^\M(X,Y) = 0$.
Then $\II_{\S,\M}$ and $\nabla^\perp$ satisfy the Gauss-Ricci-Codazzi equations with zero left-hand side, hence by \cite[Section 3]{Ten71} there exist, locally, smooth $F,\qperp$ as required, and are unique up to a rigid motion.\footnote{The main theorem in \cite{Ten71} only states the uniqueness of $F$, however its proof (specifically, the last paragraph on p.~34) shows the uniqueness of $\qperp$ as well.} 
Finally, if $\S$ is simply connected, then $F$ and $\qperp$ can be chosen on whole $\S$ (see remark at the end of \cite{Ten71}, or \cite[Section 3.2]{Che00}).
\end{proof}
%%%%%%

%%%%%%%%%%%%%%%%%%%%%%%%%%%%%%%%%%%%%%%%%%%%%%%%%%%%%%
\subsection{Proofs regarding the scaling $\inf E_{\S_h} = O(h^4)$}

\label{sec:proofs_h_4}
In this section we prove the results concerning the $h^4$ energy scaling; namely, that $\inf E_{\S_h} = o(h^2)$ implies $\inf E_{\S_h} = O(h^4)$ (thus completing the proof of part 2 of Theorem~\ref{thm:energy_scaling_general}) and that $\inf h^{-4}E_{\S_h}$ is bounded from below by an integral of the curvature along $\S$ (part 4 of Theorem~\ref{thm:energy_scaling_general}).

%%%%%%
\begin{lemma}
\label{lm:o_h_2_O_h_4}
If $\inf E_{\S_h} = o(h^2)$, then there exists a sequence of maps $u_h\in W^{1,2}(\S_h;\R^n)$ such that $E_{\S_h}[u_h] < Ch^4$ for some constant $C>0$ depending on $(\M,\g)$.
\end{lemma}

\begin{proof}
This follows from the analysis in \cite[Proposition~6.3]{KS14}. 
Indeed, $\inf E_{\S_h} = o(h^2)$ implies $\min E_\S = 0$. 
Therefore, by Lemma~\ref{lm:min_E_S_smooth}, there exists smooth $F:\S\to \R^n$ and $q^\perp: S\to \NS^*\otimes \R^n$ such that $dF\oplus q^\perp \in \SO{\g,\euc}$ and $\nabla q^\perp = -dF\circ \II_{\S,\M}$.

Using the coordinates and index conventions of Lemma~\ref{lm:min_E_S_smooth}, define $u_h(x^i,x^a) = F(x^i) + q^\perp(x^i,x^a\pl_a)$ (this is the coordinate equivalent of the recovery sequence \cite[Equation~(6.1)]{KS14}).
The analysis in the proof of \cite[Proposition~6.3]{KS14} implies that
\[
\dist(du_h,\SO{\g,\euc}) = (|\nabla q^\perp| + |du_h|)O(h^2) = O(h^2),
\]
where the second equality follows from the fact that $q^\perp$ and $u_h$ are uniformly bounded in $C^1$.
Therefore,
\[
E_{\S_h}[u_h] = \dashint_{\S_h} \dist^2(du_h,\SO{\g,\euc})\,d\Volg \le \dashint_{\S_h} Ch^4 \,d\Volg = Ch^4.
\]

\end{proof}

%%%%%%%

For the proof of Lemma~\ref{lm:o_h_4} below, we need the following immediate corollary of Theorem~\ref{thm:curvature_norm_asymptotic}:
\begin{corollary}
\label{cor:curvature_norm_asymptotic_uniform}
Let $K\subset \M$ be compact. Then
\beq
\label{eq:curvature_norm_asymptotic_uniform}
\lim_{h\to 0} \sup_{q\in K} \Abs{\inf_{u\in W^{1,2}(B_h(q);\R^n)} \frac{1}{h^4} E_{B_h(q)}[u] - |\Riem_q|^2} = 0.
\eeq
\end{corollary}

\begin{proof}
Assume, for the sake of contradiction, that \eqref{eq:curvature_norm_asymptotic_uniform} does not hold.
Then there exist $\e>0$ and a sequence $h_i\to 0$ and $p_i\in K$ such that for every $i$,
\[
\Abs{\inf \frac{1}{h_i^4} E_{B_{h_i}(p_i)} - |\Riem_{p_i}|^2} > \e.
\]
Since $K$ is compact, we can assume that $p_i\to p\in K$.
This contradicts Theorem~\ref{thm:curvature_norm_asymptotic_refined}, since $|\Riem_{p_i}|\to |\Riem_{p}|$.
\end{proof}

%%%%%%%
\begin{lemma}
\label{lm:o_h_4}
\[
\liminf_{h \to 0} \brk{ \inf h^{-4}E_{\S_h}} \ge c\dashint_S |\Riem^\M|^2 \,d\VolgS,
\]
where $|\cdot|$ is the norm defined in Theorem~\ref{thm:curvature_norm_asymptotic} and $c$ is a universal constant.
\end{lemma}

\begin{proof}
First, we recall that the map $p \mapsto |\Riem^\M_{p}|$ is continuous.
Fix $\e>0$, and let $\{V^i\}_{i=1}^m$ be a partition of $\S$ into small regulars sets (e.g.~embedded regular simplices) such that 
\[
\frac{1}{\VolgS(\S)} \sum_{i}^m \VolgS(V^i) |\Riem^\M_{p_i}|^2 > \dashint_S |\Riem^\M|^2 \,d\VolgS - \e
\]
for some $p_i\in V^i$.
We can, furthermore, choose $V^i$ small enough such that for every $q\in V^i$, $|\Riem^\M_{q}|^2\ge |\Riem^\M_{p_i}|^2- \e$.
For $h$ small enough, denote $V^i_h = \pi_h^{-1}\brk{V^i}$.
Assuming $V^i$ is regular enough, there exists $h_\e$ (depending on the partition), such that For $h<h_\e$ we can choose disjoint balls $\{B_{h}(q^{i,j}_h)\}_{j=1}^{n^i_h}$ of radius $h$, centered at $q^{i,j}_h\in V^i$, such that $B_{h}(q^{i,j}_h)\subset V^i_h$ and 
\[
\sum_{j=1}^{n_h} \Volg(B_{h}(q^{i,j}_h)) \ge c\Volg(V^i_h)
\]
for some universal constant $c>0$ independent of $\e$, $i$, $h$ and $\S$.
Now, for a given $u_h\in  W^{1,2}(\S_h;\R^n)$, we have
\[
\begin{split}
E_{\S_h}[u_h] 
	&= \frac{1}{\Volg(\S_h)}\sum_{i=1}^m \int_{V^i_h} \dist^2(du_h, \SO{\g,\euc})\, d\Volg \\
	&\ge \frac{1}{\Volg(\S_h)}\sum_{i=1}^m \sum_{j=1}^{n_h^i} \int_{B_{h}(q^{i,j}_h)} \dist^2(du_h, \SO{\g,\euc})\, d\Volg \\
	&= \frac{1}{\Volg(\S_h)}\sum_{i=1}^m\sum_{j=1}^{n_h^i} \Volg(B_{h}(q^{i,j}_h))\, E_{B_{h}(q^{i,j}_h)}[u_h]\\
	&\ge \frac{1}{\Volg(\S_h)}\sum_{i=1}^m\sum_{j=1}^{n_h^i} \Volg(B_{h}(q^{i,j}_h))\, \inf E_{B_{h}(q^{i,j}_h)}.\end{split}
\] 
Using Theorem~\ref{thm:curvature_norm_asymptotic} we then have
\[
\begin{split}
E_{\S_h}[u_h] 
	&\ge \frac{1}{\Volg(\S_h)}\sum_{i=1}^m\sum_{j=1}^{n_h^i} \Volg(B_{h}(q^{i,j}_h))\, \inf E_{B_{h}(q^{i,j}_h)}\\
	&\ge \frac{1}{\Volg(\S_h)}\sum_{i=1}^m\sum_{j=1}^{n_h^i} \Volg(B_{h}(q^{i,j}_h))\,\brk{h^4 |\Riem^\M_{q^{i,j}_h}|^2 + o(h^4)}\\
	&\ge \frac{1}{\Volg(\S_h)}\sum_{i=1}^m\sum_{j=1}^{n_h^i} \Volg(B_{h}(q^{i,j}_h))\,\brk{h^4\brk{|\Riem^\M_{p_i}|^2-\e} + o(h^4)}\\
	&\ge c h^4\frac{1}{\Volg(\S_h)}\sum_{i=1}^m \brk{|\Riem^\M_{p_i}|^2-\e} \Volg(V^i_h) + o(h^4) \\
	&= c h^4\brk{\frac{1}{\Volg(\S_h)}\sum_{i=1}^m |\Riem^\M_{p_i}|^2 \Volg(V^i_h) - \e} + o(h^4),
\end{split}
\] 
where we used Corollary~\ref{cor:curvature_norm_asymptotic_uniform} for $K=\S$ to take the $o(h^4)$ term uniformly with respect to $q^{i,j}_h$ in last line.
Now, using the fact that $\Volg(V^i_h) = h^{n-k}\VolgS(V^i)(1+o(1))$ and $\Volg(\S_h) = h^{n-k}\VolgS(\S)(1+o(1))$, we have
\[
\begin{split}
E_{\S_h}[u_h] 
	&\ge c h^4\brk{\frac{1}{\Volg(\S_h)}\sum_{i=1}^m |\Riem^\M_{p_i}|^2 \Volg(V^i_h) -\e}+ o(h^4)\\
	&= c h^4 \brk{\frac{1}{\VolgS(\S)}\sum_{i=1}^m |\Riem^\M_{p_i}|^2 \VolgS(V^i) -\e} + o(h^4)\\
	&\ge c h^4\brk{\dashint_\S |\Riem^\M|^2 \,d\VolgS - 2\e} + o(h^4).
\end{split}
\] 
Taking the infimum over $u_h$, dividing by $h^{4}$ and taking the limit $h\to 0$, we then have
\[
\liminf \brk{\inf h^{-4}E_{\S_h}} \ge c \brk{\dashint_\S |\Riem^\M|^2 \,d\VolgS - 2\e}.
\]
Since $\e$ is arbitrary, the proof is complete.
\end{proof}

%%%%%%%%%%%%%%%%%%%%%%%%%%%%%%%%%%%%%%%%%%%%%%
{\footnotesize
\newcommand{\etalchar}[1]{$^{#1}$}
\providecommand{\bysame}{\leavevmode\hbox to3em{\hrulefill}\thinspace}
\providecommand{\MR}{\relax\ifhmode\unskip\space\fi MR }
\providecommand{\MRhref}[2]{%
  \href{http://www.ams.org/mathscinet-getitem?mr=#1}{#2}
}
\providecommand{\href}[2]{#2}

\bibliographystyle{amsalpha}

\begin{thebibliography}{AKM{\etalchar{+}}16}

\bibitem[AAE{\etalchar{+}}12]{AAESK12}
H.~Aharoni, Y.~Abraham, R.~Elbaum, E.~Sharon, and R.~Kupferman, \emph{Emergence
  of spontaneous twist and curvature in {non-Euclidean} rods: Application to
  Erodium plant cells}, Phys. Rev. Lett. \textbf{108} (2012), 238106.

\bibitem[AESK11]{AESK11}
S.~Armon, E.~Efrati, E.~Sharon, and R.~Kupferman, \emph{Geometry and mechanics
  of chiral pod opening}, Science \textbf{333} (2011), 1726--1730.

\bibitem[AKM{\etalchar{+}}16]{AKMMS16}
H.~Aharoni, J.~M.~Kolinski, M.~Moshe, I.~Meirzada, and E.~Sharon, \emph{Internal stresses lead to net forces and torques on extended
  elastic bodies}, Phys. Rev. Lett. \textbf{117} (2016), 124101.

\bibitem[ALL17]{ALL17}
V.~Agostiniani, A.~Lucantonio, and D.~Lu{\v{c}}i{\'{c}}, \emph{Heterogeneous
  elastic plates with in-plane modulation of the target curvature and
  applications to thin gel sheets}, preprint, 2017.

\bibitem[BBS55]{BBS55}
B.A. Bilby, R.~Bullough, and E.~Smith, \emph{Continuous distributions of
  dislocations: A new application of the methods of {Non-Riemannian} geometry},
  Proc. Roy. Soc. A \textbf{231} (1955), 263--273.

\bibitem[BK14]{BK14}
P.~Bella and R.V.~Kohn, \emph{Metric-induced wrinkling of a thin
  elastic sheet}, Journal of Nonlinear Science \textbf{24} (2014), no.~6,
  1147--1176.

\bibitem[BLS16]{BLS15}
K.~Bhattacharya, M.~Lewicka, and M.~Sch\"affner, \emph{Plates with
  incompatible prestrain}, Arch. Rational Mech. Anal. \textbf{221} (2016),
  no.~1, 143--181.

\bibitem[BS56]{Bs56}
B.A. Bilby and E.~Smith, \emph{Continuous distributions of dislocations.
  {III}}, Proc. Roy. Soc. Edin. A \textbf{236} (1956), 481--505.

\bibitem[Che00]{Che00}
B.Y. Chen, \emph{Riemannian submanifolds: A survey}, Handbook of Differential
  Geometry, edited by F. Dillen and L. Verstraelen \textbf{1} (2000), 187--418.

\bibitem[Cia88]{Cia88}
P.~G. Ciarlet, \emph{Mathematical elasticity, volume 1: Three-dimensional
  elasticity}, Elsevier, 1988.

\bibitem[Cia05]{Cia05}
\bysame, \emph{An introduction to differential geometry with applications to
  elasticity}, Springer Netherlands, 2005.

\bibitem[Cia13]{Cia13}
\bysame, \emph{Linear and nonlinear functional analysis with
  applications}, SIAM, 2013.

\bibitem[COT17]{COT17}
S.~Conti, H.~Olbermann, and I.~Tobasco, \emph{Symmetry breaking in
  indented elastic cones}, Mathematical Models and Methods in Applied Sciences
  \textbf{27} (2017), no.~2, 291--321.

\bibitem[CRS17]{CRS17}
M.~Cicalese, M.~Ruf, and F.~Solombrino, \emph{On global and
  local minimizers of prestrained thin elastic rods}, Calculus of Variations
  and Partial Differential Equations \textbf{56} (2017), no.~4, 115.

\bibitem[dC92]{doC92}
M.P.~do~Carmo, \emph{Riemannian geometry}, Birkh{\"a}user, 1992.

\bibitem[ESK09a]{ESK09}
E.~Efrati, E.~Sharon, and R.~Kupferman, \emph{Buckling transition and boundary
  layer in non-{Euclidean} plates}, PRE \textbf{80} (2009), 016602.

\bibitem[ESK09b]{ESK08}
\bysame, \emph{Elastic theory of unconstrained non-{Euclidean} plates}, Journal
  of the Mechanics and Physics of Solids \textbf{57} (2009), 762--775.

\bibitem[ESK11]{ESK11}
\bysame, \emph{Hyperbolic non-{Euclidean}
  elastic strips and almost minimal surfaces}, PRE \textbf{83} (2011), 046602.

\bibitem[FJM02]{FJM02b}
G.~Friesecke, R.D.~James, and S.~M\"uller, \emph{A theorem on geometric
  rigidity and the derivation of nonlinear plate theory from three dimensional
  elasticity}, Comm. Pure Appl. Math. \textbf{55} (2002), 1461--1506.

\bibitem[FJM06]{FJM06}
\bysame, \emph{A hierarchy of plate models derived from nonlinear elasticity by
  {$\Gamma$}-convergence}, Arch. Rat. Mech. Anal. \textbf{180} (2006),
  183--236.

\bibitem[GSD16]{GSD16}
D.~Grossman, E.~Sharon, and H.~Diamant, \emph{Elasticity and
  fluctuations of frustrated nanoribbons}, Phys. Rev. Lett. \textbf{116}
  (2016), 258105.

\bibitem[KES07]{KES07}
Y.~Klein, E.~Efrati, and E.~Sharon, \emph{Shaping of elastic sheets by
  prescription of non-{Euclidean} metrics}, Science \textbf{315} (2007), 1116
  -- 1120.

\bibitem[KM14]{KM14}
R.~Kupferman and C.~Maor, \emph{A {Riemannian} approach to the membrane limit
  in non-{Euclidean} elasticity}, Comm. Contemp. Math. \textbf{16} (2014),
  no.~5, 1350052.

\bibitem[KMS]{KMS17}
R.~Kupferman, C.~Maor, and A.~Shachar, \emph{Asymptotic rigidity of
  {Riemannian} manifolds}, \url{https://arxiv.org/abs/1701.08892}.

\bibitem[KO18]{KO18}
Robert~V. Kohn and Ethan O'Brien, \emph{On the bending and twisting of rods
  with misfit}, Journal of Elasticity \textbf{130} (2018), no.~1, 115--143.

\bibitem[Kon55]{Kon55}
K.~Kondo, \emph{Geometry of elastic deformation and incompatibility}, Memoirs
  of the Unifying Study of the Basic Problems in Engineering Science by Means
  of Geometry (K.~Kondo, ed.), vol.~1, 1955, pp.~5--17.

\bibitem[KS12]{KS12}
R.~Kupferman and Y.~Shamai, \emph{Incompatible elasticity and the immersion
  of non-flat {Riemannian} manifolds in euclidean space}, Israel Journal of Mathematics \textbf{190} (2012), no.~1, 135--156.

\bibitem[KS14]{KS14}
R.~Kupferman and J.P.~Solomon, \emph{A {Riemannian} approach to reduced
  plate, shell, and rod theories}, Journal of Functional Analysis \textbf{266}
  (2014), 2989--3039.

\bibitem[LDR95]{LR95}
H.~Le-Dret and A.~Raoult, \emph{The nonlinear membrane model as a variational
  limit of nonlinear three-dimensional elasticity}, Journal de Mathematiques
  Pures et Appliquees \textbf{74} (1995), 549--578.

\bibitem[LDR96]{LR96}
\bysame, \emph{The membrane shell model in nonlinear elasticity: A variational
  asymptotic derivation}, Journal of Nonlinear Science \textbf{6} (1996),
  no.~1, 59--84.

\bibitem[LP11]{LP10}
M.~Lewicka and M.R.~Pakzad, \emph{Scaling laws for non-{Euclidean} plates and
  the {$W^{2,2}$} isometric immersions of {Riemannian} metrics}, ESAIM:
  Control, Optimisation and Calculus of Variations \textbf{17} (2011),
  1158--1173.

\bibitem[LRR]{LRR15}
M.~Lewicka, A.~Raoult, and D.~Ricciotti, \emph{Plates with incompatible
  prestrain of higher order}, To appear in {Annales de l'Institut Henri
  Poincare (C) Non Linear Analysis}.

\bibitem[Olb17]{Olb17}
H.~Olbermann, \emph{Energy scaling law for a single disclination in a thin
  elastic sheet}, Arch. Rat. Mech. Anal. \textbf{224} (2017), no.~3, 985--1019.

\bibitem[OY09]{OY09}
A.~Ozakin and A.~Yavari, \emph{A geometric theory of thermal stresses}, J.
  Math. Phys. \textbf{51} (2009), 032902.

\bibitem[SRS07]{SRS07}
E.~Sharon, B.~Roman, and H.L.~Swinney, \emph{Geometrically driven wrinkling
  observed in free plastic sheets and leaves}, PRE \textbf{75} (2007), 046211.

\bibitem[Ten71]{Ten71}
K.~Tenenblat, \emph{On isometric immersions of {Riemannian} manifolds}, Boletim
  da Soc. Bras. de Mat. \textbf{2} (1971), 23--36.

\end{thebibliography}
}

\end{document}